\documentclass[leqno, 12pt]{amsart}
\usepackage[utf8]{inputenc}
\usepackage{lmodern}
\usepackage[english]{babel}
\usepackage{amstext}
\usepackage{amsthm}
\usepackage{amsmath}
\usepackage{amssymb}
\usepackage{graphicx}
\usepackage{mathtools}
\usepackage[svgnames]{xcolor}
\definecolor{color1}{RGB}{27,158,119}
\definecolor{color2}{RGB}{217,95,2}
\definecolor{color3}{RGB}{117,112,179}
\definecolor{color4}{RGB}{231,41,138}
\usepackage[margin=1in]{geometry}
\usepackage{dsfont}
\usepackage{fancyvrb}
\usepackage[unicode=true,pdfusetitle,
 bookmarks=true,bookmarksnumbered=false,bookmarksopen=false,
 breaklinks=false,pdfborder={0 0 1},backref=false,colorlinks=true]
 {hyperref}

\hypersetup{
  pdftitle={Ground states on a fractured strip and one-dimensional reduction},
  pdfauthor={Stefan Le Coz  and Boris Shakarov},
  pdfsubject={},
  pdfkeywords={}, 
  colorlinks=true,
  linkcolor=DarkBlue  ,          %
  citecolor=DarkRed,        %
  filecolor=DarkMagenta,      %
  urlcolor=DarkGreen,           %
}

\newtheorem{theorem}{Theorem}[section]
\newtheorem{lem}[theorem]{Lemma}
\newtheorem{prop}[theorem]{Proposition}
\newtheorem{proposition}[theorem]{Proposition}

\newtheorem{remark}[theorem]{Remark}
\theoremstyle{definition}

\def\R{\mathbb R}

\def\N{\mathbb N}

\def\Z{\mathbb Z}
\def\T{\mathbb T}

\def\eps{\varepsilon}

\newcommand{\Strip}{\mathbb S}
\newcommand{\deltastrip}{\delta_0(x)}
\newcommand{\deltastripunit}{\delta_0(x)}

\DeclareMathOperator{\sech}{sech}
\DeclareMathOperator{\supp}{supp}
\DeclarePairedDelimiter{\norm}{\lVert}{\rVert}

\DeclarePairedDelimiterX{\dual}[2]{\langle}{\rangle}{#1, #2}

\title[Ground states on a fractured strip and $1$-$d$  reduction]{Ground states on a fractured strip \\and one dimensional reduction}

\author[S. Le Coz]{Stefan Le Coz}
\address{Institut de Mathématiques de Toulouse ; UMR5219, Université de Toulouse ; CNRS, UPS IMT, F-31062 Toulouse Cedex 9 (France)}
\email{stefan.le-coz@math.univ-toulouse.fr}

\author[B. Shakarov]{Boris Shakarov}
\address{Institut de Mathématiques de Toulouse ; UMR5219, Université de Toulouse ; CNRS, UPS IMT, F-31062 Toulouse Cedex 9 (France)}
\email{boris.shakarov@math.univ-toulouse.fr}
\thanks{The work of S. L. C. and B. S. is 
  partially supported by ANR-11-LABX-0040-CIMI and the ANR project NQG ANR-23-CE40-0005}

\date{\today}


\date{\today}

\begin{document}

\begin{abstract}
 We consider the nonlinear Schr\"odinger equation on a strip with Neumann boundary conditions and a delta condition on the $x$-axis. First, we show the existence of ground states as minimizers of the action or of the energy under suitable constraints. Second, we prove that the energy minimizers converge to the ground state on the line with a delta condition as the amplitude of the strip shrinks to zero.

\end{abstract}

\maketitle

\section{Introduction}

In many modeling situations, it is common to consider one-dimensional models as simplified versions of higher-dimensional models, with the belief that the $1$-$d$ model will keep the major features of the higher-dimensional model while being simpler to analyze. While it indeed often appears that $1$-$d$ models are more amenable to analysis, the connection with their higher dimensional counterparts is rarely investigated.  

Among simplified $1$-$d$ models, quantum graphs stand out and have been the subject of intense analysis. Higher dimensional counterparts of quantum graphs can be constructed in the form of \emph{fat graphs} (also called graph-like manifolds, tubular branched manifolds, inflated graphs, etc., see e.g. \cite{Ex08,Po12}), but the connections between the two types of models are highly non-trivial, even in linear situations. 

In recent years, nonlinear quantum graphs, i.e. quantum graphs endowed with a nonlinear Schrödinger equation, have been the object of extensive investigations, in particular for the question of the existence of minimizers for the energy at fixed mass. However, to the authors' knowledge, there are no studies establishing connections of nonlinear quantum graphs with higher dimensional counterparts, with the notable exception of the works of Kosugi \cite{Ko00,Ko02}. In \cite{Ko00,Ko02}, Kosugi established the convergence of solutions to a semilinear equation on thin network-shaped domains to their counterparts on the network. The works \cite{Ko00,Ko02} are devoted to compact domains, and the assumptions made on the nonlinearity exclude the common model nonlinearities for nonlinear quantum graphs such as the power type nonlinearity. 
Several related studies are devoted to product spaces. Terracini, Tzvetkov, and Visciglia  \cite{TeTzVi14} considered the nonlinear Schr\"odinger equation on product spaces of the type $\mathbb R^d\times \mathcal M$, where $\mathcal M$ is a compact manifold. In  \cite{TeTzVi14}, they have established a connection between energy/mass minimizers on the product space with ground states on $\mathbb R^d$. 
The Gross-Pitaevskii equation on $\mathbb R\times \mathbb T$ with a general nonlinearity was recently considered by Mariş and Mur in \cite{MaMu24}, where it was shown in particular that energy/momentum minimizers become one dimensional when the size of the torus $\mathbb T$ shrinks to $0$ (see also \cite{deGrSm24} for the cubic nonlinearity case). The bifurcation from a $1$-$d$ line soliton and a full $2$-$d$ soliton on the cylinder $\R\times \mathbb T$ was studied by Akahori, Bahri, Ibrahim, and Kikuchi in \cite{AkBaIbKi24} (see also \cite{Ya15} for earlier work).

The present paper aims to start the exploration of the connection between nonlinear quantum graphs and higher dimensional equivalents. We consider the simplest possible setting beyond the strip: a fractured strip, modeled by a strip with a $\delta$ distribution in the transverse direction at the origin. In the shrinking limit, such a strip converges formally towards the line with a $\delta$ distribution at the origin, which might be interpreted as a $2$-star graph with $\delta$ vertex condition. 

Let $L>0$. We denote the strip with width $L$ as $\Strip_L$, that is $\Strip_L = \R \times [0,L]$ with the convention that $\Strip = \Strip_{L=1}$. 
We search for solutions to the equation 
\begin{equation}\label{eqDeltaStrip} 
\begin{cases}
     - \partial_{xx} u - \partial_{yy} u + \omega u + \gamma \deltastrip u - |u|^{p-1} u = 0, \quad (x,y) \in \Strip_L,\\
     \partial_\nu u = 0\, \mbox{ on } \partial\Strip_L= \R \times \{0,L\},
\end{cases}  
\end{equation}
in $H^1(\Strip_L)$, where $\omega>0$, $\gamma\in\R$, the symbol $\deltastrip  \in H^{-1}(\Strip_L)$ is defined 
for $u,v\in H^{1}(\Strip_L)$ by 
\[
(\deltastrip  u, v):= \int_0^L Re\left(u(0,y)\overline{v(0,y)} \right)dy,
\]
and $\partial_\nu$ denotes the normal derivative on the borders and could be replaced by $\partial_y$ in the present context. By the trace theorem (see e.g. \cite[p. 315]{Brezis2011} or Section \ref{secPreliminar}), $\deltastrip $ is well defined. 
We will refer to the case $\gamma >0$ as the \emph{repulsive case}, and $\gamma <0$ as the \emph{attractive case}. Solutions of \eqref{eqDeltaStrip} are regular and exponentially decaying  (see Appendix \ref{sec:properties}).

The first part of this work is dedicated to finding solutions to \eqref{eqDeltaStrip}. The existence of positive solutions to \eqref{eqDeltaStrip} can be addressed by variational methods in at least two different ways, which we now describe. For any $u\in H^1(\Strip_L)$ and $\omega \in \R$, we define the action, the Nehari functional, the energy, and the mass  by
\begin{align}    
    \label{eqAction2D}
    S_{\omega,\gamma}(u) &= \frac{1}{2} \| \nabla u \|_{L^2(\Strip_L)}^2 + \frac{\omega}{2} \| u \|_{L^2(\Strip_L)}^2 + \frac{\gamma}{2} \int_0^L |u(0,y)|^2 dy - \frac{1}{p+1} \| u \|_{L^{p+1}(\Strip_L)}^{p+1},
\\
    \label{eqNehari2D}
   I_{\omega,\gamma}(u)&=  \| \nabla u \|_{L^2(\Strip_L)}^2 + \omega \| u \|_{L^2(\Strip_L)}^2 + \gamma \int_0^L |u(0,y)|^2 dy - \| u \|_{L^{p+1}(\Strip_L)}^{p+1},
\\
\label{eqEnergy2D}
    E_\gamma(u) &= \frac{1}{2} \| \nabla u \|_{L^2(\Strip_L)}^2 + \frac{\gamma}{2} \int_0^L |u(0,y)|^2 dy - \frac{1}{p+1} \| u \|_{L^{p+1}(\Strip_L)}^{p+1},
\\
    \label{eqMass2D}
     M(u) &= \| u\|_{L^2(\Strip_L)}^2.%
\end{align}
Sobolev embeddings and the trace theorem ensure that the functionals above are well defined for any $u \in H^1(\Strip_L)$. We consider the following minimization problems
\begin{gather} 
\label{eqActNehIntr}
    s_{\omega,\gamma} = \inf\{S_{\omega,\gamma}(u): u \in H^1 (\Strip_L) \setminus\{0\}, \,I_{\omega,\gamma}(u)= 0\},
\\
\label{eqMinEnergyMass}
    e_{m,\gamma} = \inf \{E_\gamma(u) : u \in H^1(\Strip_L), \, M(u) = m \},\quad m>0.
\end{gather}
The minimizers of \eqref{eqActNehIntr} are referred to as \emph{action ground states} and the set $\{u\in H^1(\Strip_L):I(u)=0\}$ is called \emph{Nehari manifold}. The minimizers of \eqref{eqMinEnergyMass} are called \emph{energy ground states}; in this context, $\omega$ appears as a Lagrange multiplier. In physical settings, it is often interesting to obtain energy ground states, which also have the advantage of being expected to be stable for the Schr\"odinger dynamics. However, it has long been known since the seminal works of Berestycki, Cazenave, and Lions \cite{BeCa81,CaLi82} that energy ground states do not exist for supercritical power (i.e $p$ larger than a threshold depending on the dimension of the problem). On the other hand, action ground states should exist for any $H^1$-subcritical $p$, at least for $\omega$ in a given frequency range, but, their stability is, a priori, undetermined.

The relationship between the energy and the action ground states is not yet fully understood. It is established that generically all energy ground states are also action ground states. However, the conditions under which the reverse is true remain unclear. For further details, we refer to the works of Dovetta, Serra, and Tilli \cite{DoSeTi23}, Jeanjean and Lu \cite{JeLu22}, as well as De Coster, Dovetta, Galant and Serra \cite{DeDoGaSe23}.

We will show the following in the attractive case $\gamma <0$.
\begin{theorem}
\label{thm:action-attract}
    Let $\gamma <0$, $p >1$, and $\omega > \frac{\gamma^2}{4}$. Then there exists a real-valued and positive action ground state, i.e., a minimizer for \eqref{eqActNehIntr}. This minimizer satisfies Equation \eqref{eqDeltaStrip}.
\end{theorem}
\begin{theorem}
\label{ThmEnMinIntro}
    Let  $\gamma <0$, $1 < p < 3$, and $m >0$. Then there exists a real-valued and positive energy ground state, i.e. a minimizer %
    for \eqref{eqMinEnergyMass}. Moreover, there exists $\omega = \omega(m) > \frac{\gamma^2}{4}$ such that this minimizer satisfies Equation \eqref{eqDeltaStrip}.
\end{theorem}

The existence of minimizers in the repulsive case $\gamma >0$ requires additional assumptions due to the possibility of the \textit{run-away} behavior, i.e. minimizing sequences escaping on one side of the strip. This behavior also occurs in the one-dimensional case, which is why results in that case are restricted to minimization over radial (even) functions (see Proposition \ref{prp1DIntr}). On the strip, we define the space of functions symmetric with respect to the origin in the $x$ variable by
\begin{equation}
\label{eqH1Symmetric}
     H^1_{sym}(\Strip_L) = \{ f \in H^1(\Strip_L)\, : \forall x \in \R, \, f(x,y) = f(-x,y)\},
\end{equation}
and the corresponding minimization problems are denoted in the following way:
\begin{gather}\label{eqActNehRadIntr}
    s_{\omega,\gamma}^{sym}
= \inf\{S_{\omega,\gamma}(u), \, u \in H^1_{sym}  (\Strip_L) \setminus\{0\}, \,I_{\omega,\gamma}(u)= 0\},
\\
\label{eqEnMasRadIntr}
e_{m,\gamma}^{sym}
    = \inf \{E_\gamma(u) : u \in H^1_{sym}(\Strip_L), \, M(u) = m \},\quad m>0.
\end{gather}
In one dimension, explicit expressions for the minimizers could be obtained by simple surgery from the formula of the ground state on the line. No such feature is present in higher dimensions, and this generates additional complications in the study of the repulsive case. We obtained the existence of action minimizers for small values of the parameters $L>0$ or $\gamma>0$. 

\begin{theorem}
    \label{thm:rep_action}
     Let $\gamma > 0$,  $p >1$, and $\omega > \frac{\gamma^2}{4}$. There exists $\gamma^*=\gamma^*(\omega, L)>0$ and $L^\dagger=L^\dagger(\omega)$ such that if $ 0 < \gamma < \gamma^*$ or $L<L^\dagger$, 
    then there exists a real-valued and positive symmetric action ground state, i.e. a minimizer for \eqref{eqActNehRadIntr}. This minimizer satisfies Equation \eqref{eqDeltaStrip}.
\end{theorem}
    In the statement of Theorem \ref{thm:rep_action}, the threshold parameter $\gamma^*$ depends on $L$, but the threshold parameter $L^\dagger$ does not depend on $\gamma$.  We provide different proofs to cover each case. 

Theorem \ref{thm:rep_action} deals with problem \eqref{eqActNehRadIntr}. As similar result is expected to hold for problem \eqref{eqEnMasRadIntr}. 
We do not provide a precise statement for \eqref{eqEnMasRadIntr} and refer to  Section \ref{secEnMinRepul} for a discussion of this problem.

The minimizers obtained in the previous results are regular, exponentially decaying, symmetric with respect to the line $y=0$ and monotonic in $y$ (see Appendix \ref{sec:properties}).

In the second part of this work, our scope is two-fold. On the one hand, we study the behavior of the ground states when the amplitude $L$ shrinks to $0$. In this case, we will show that there exists a threshold $L^*>0$ such that for any $0 < L < L^*$ the non-negative ground state does not depend on the transverse variable $y$.

\begin{theorem}\label{thmShrink}
    Let $\gamma < 0$. Then for any $\tilde m >0$, there exists $L^*=L^*(\tilde m) >0$ such that for any $0 < L < L^*$ the energy minimizer with mass $m=\tilde m L$ found in Theorem \ref{ThmEnMinIntro} does not depend on the transverse variable $y$.
\end{theorem}

On the other hand, we will also show that there exists a second threshold $L^{**}$ such that for $L > L^{**}$ the ground state is truly two-dimensional, as a one-dimensional profile will always be energetically unfavorable. 

\begin{theorem}\label{thmShrink2}
     Let $\gamma < 0$,
     Then for any $\tilde m >0$, there exists $L^{**}=L^{**}(\tilde{m}) >0$ such that for any $L > L^{**}$ the energy minimizer with mass $m=\tilde m L$ found in Theorem \ref{ThmEnMinIntro}  depends on the transverse variable $y$ in a nontrivial way.
     \end{theorem}

The proof of Theorem \ref{thmShrink} relies on a delicate rigidity argument for the minimizers on the strip. To work in comparable settings, we renormalized the problems in such a way that the length of the strip is fixed (the parameter $L$ appearing now in the energy). We then establish convergence of minimizers on the strip to the extended $1$-$d$ minimizers when $L\to 0$, and derive the rate of convergence. The rigidity argument consists then in showing that the obtained rate of convergence implies that the solutions do not depend on the transverse variable for small $L$.

Minimization problems are very amenable to numerical simulations. We performed several numerical experiments which agreed with the theoretical results obtained in the present work. The full details of the algorithms and experiments will be reported in future work. We give two illustrations of the outcome of a minimization algorithm for the action over the Nehari manifold in the case of a short amplitude strip and a wider amplitude strip, see Figure \ref{fig:illustration}. We observe that in the short amplitude case, the numerical minimizer is indeed one dimensional, while in the large amplitude case, the numerical minimizer is truly two-dimensional, symmetric with respect to $x=0$, monotonic in $y$, and decaying fastly in $x$. These numerical observations confirm the theoretical results. 

\begin{figure}[htpb!]
    \centering
    \includegraphics[width=0.5\linewidth]{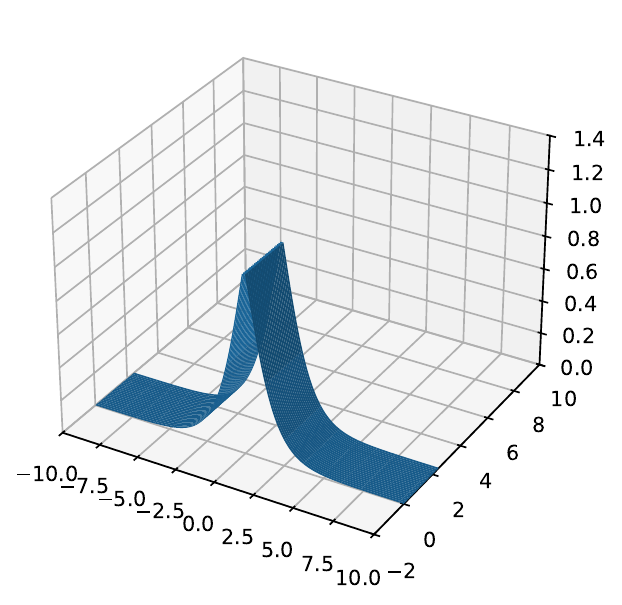}
    \includegraphics[width=0.49\linewidth]{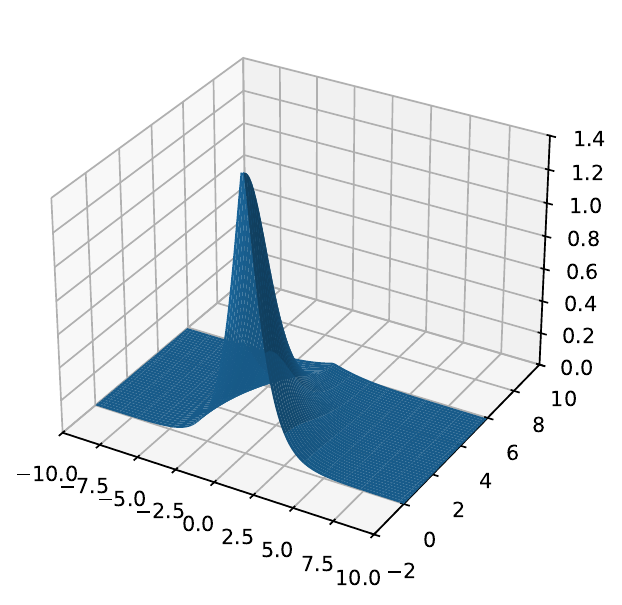}
    \caption{Outcomes of numerical minimization over the Nehari manifold when $\gamma=-1$ for a strip of length $2.5$ (left picture) and length $6$ (right picture)}
    \label{fig:illustration}
\end{figure}

This work is structured as follows. Section \ref{secPreliminar} recalls preliminary results and introduces notation. In Section \ref{secActionMin} we find solutions to \eqref{eqDeltaStrip} as action minimizers on the Nehari constraint. In Section \ref{secEnergyMin}, we show the existence of energy minimizers under the mass constraint. In Section \ref{secStripToLine} we study the limits as the amplitude of the strip shrinks to zero or grows to infinity. In Appendix \ref{sec:properties}, we establish qualitative properties of solutions of \eqref{eqDeltaStrip} and action and energy ground states.

\section{Preliminaries}\label{secPreliminar}
In this section, we introduce the notation and recall results used throughout this work. 

Functions in $L^1_{loc}(\Omega)$ (for $\Omega\subset \R^d$) will always be substituted by their precise representative, where the \emph{precise representative} (see e.g. \cite{EvGa15}) of $f\in L^1_{loc}(\Omega)$  is the function $f^*$ defined for every $x\in \Omega$ by
\[
f^*(x)=
\begin{cases}
    \lim_{r\to0}\frac{1}{|B(x,r)|}\int_{B(x,r)}f(y)dy& \text{ if the limit exists,}\\
    0& \text{ otherwise}.
\end{cases}
\]

\subsection{The trace theorem}
In this section, we define rigorously the traces that are used along the work and present the properties needed for the proofs.  
We start by defining the trace of the functions projected on the hyperplane $x = 0$. Let us denote by 
\begin{equation*}
    (\tau f)(x,y) = f(0,y) 
\end{equation*}
for $f$ a continuous function in $\Strip_L$. The trace $\tau u$ is well defined as soon as $u \in W^{s,p}(\Strip_L)$ when $s > 1/p$ and the map $\tau: W^{s,p}(\Strip_L) \to  W^{s-1/p,p}(0,L)$ is bounded, see \cite[Theorem 1.5.1.1]{Gr11} for example. For completeness, we give here a short proof of this fact when $p = 2$. 

\begin{lem}\label{lemTraceComp}
    For any $s > \frac{1}{2}$ the trace operator $\tau$ is a bounded operator $\tau: H^s(\Strip_L) \to H^{s - \frac{1}{2}}(0,L)$. Moreover,  $\tau: H^1(\Strip_L) \to L^2(0,L)$ is a compact operator. 
\end{lem}
\begin{proof}
    Given any $u\in H^1{(\Strip_L)}$ we first perform the mirrored continuation and then extend it to a periodic function in the transverse variable $y$. Precisely, define
    \begin{equation*}
        \tilde u(x,y) = \begin{cases}
            u(x,y), \ y \in [0,L], \\
            u(x,-y), \ y \in [-L, 0],
        \end{cases}
    \end{equation*}
    so that $\tilde u \in H^1(\R \times[-L,L])$ is symmetric with respect to the $x$ - axis. Then let $v$ be the periodic continuation of $u$, that is if $J_{k} = [-L + 2kL, L + 2kL)$ for $k \in \Z$ then $v(x,y) = \tilde u (x, y - 2kL)$ for any $x$ and $y \in J_k$. 
    
    Now suppose that $v \in \mathcal S := \{ v\in \mathcal C^\infty(\R^2) : \forall \alpha,\beta \geq 0, \sup_{x \in \R} |x^\alpha \partial_x^\beta v(x,y)| < \infty  \} $. We have 
    \begin{equation*}
        v(0,y) = \int_\R \int_\R e^{-2\pi i x \xi} v(x,y) dx d\xi.
    \end{equation*}
    We denote by $\hat\cdot_n$ the Fourier coefficients in $y$ and by $\tilde\cdot$ the Fourier transform in $x$. By Cauchy Schwarz inequality, we obtain
    \begin{equation*}
    \begin{aligned}
         \hat{v}_{n}(0) &= \frac{1}{L}\int_{-L}^L e^{-2\pi y  i n/L  } v(0,y) dy = \frac{1}{L} \int_\R \int_{\R} \int_{-L}^L  e^{-2\pi  iy n/L  } e^{-2\pi i x \xi} v(x,y) dy dx  d\xi \\ 
         &= \int_\R   \hat{\tilde{v}}_{n}(\xi) d\xi \leq \left( \int_{\R}  (1 + n^2 + \xi^2)^{s} |\hat{\tilde{v}}_{n}(\xi)|^2 d\xi \right)^\frac{1}{2} \left( \int_{\R} (1 + n^2 + \xi^2)^{-s}  d\xi \right)^\frac{1}{2} 
    \end{aligned}
    \end{equation*}
    for any $n \in \Z$ and $s >0$. 
    By the change of variable $\xi = t \sqrt{1 + n^2}$ we have 
    \begin{equation*}
        \int_\R (1 + n^2 + \xi^2)^{-s} d\xi =  (1 + n^2)^{1/2 - s} \int_\R \frac{1}{(1 + t^2)^s} dt \lesssim (1 + n^2)^{1/2 - s}
    \end{equation*}
    for $s > \frac{1}{2}$. Consequently, we get 
    \begin{equation*}
        \begin{aligned}
           \sum_{n \in \Z} (1 + n^2)^{s - 1/2} |\hat{v}_n(0)|^2 & \leq \sum_{n \in \Z}  \int_{\R}  (1 + n^2 + \xi^2)^{s} |\hat{\tilde{v}}_{n}(\xi)|^2 d\xi \\
           & \lesssim  \sum_{n \in \Z} (1 + n^2)^s \int_{\R}  (1 + \xi^2)^{s} |\hat{\tilde{v}}_{n}(\xi)|^2 d\xi 
        \end{aligned}
    \end{equation*}
    which implies the result for $v\in\mathcal S$.
    The conclusion for the first part of the Lemma follows by density and noticing that $\| u \|_{H^s(\Strip_L)} = \frac{1}{2} \| \tilde u \|_{H^s{(\R \times [-L, L])}} =  \frac{1}{2} \| v \|_{H^s{(\R \times \T_{2L})}}$. 

    For the compactness statement we observe that $\tau:H^1(\Strip_{L}) \to H^{1/2}(0,L)$ is bounded. By Rellich-Kondrachov theorem and Sobolev embedding theorem, the embedding $H^{1/2}([0, L]) \hookrightarrow L^{p}(0, L)$ is compact for any $p \in [1,\infty)$, which gives the desired result. 
    \end{proof}
    
    With abuse of notations, we will write $u(0,y)$ for $\tau u (x,y)$ for $u \in H^1(\Strip_L)$.  We notice the following weaker inequality which will be used in the following. 
\begin{lem}
    There exists $C > 0$ such that 
    \begin{equation}\label{eqH12Control}
        \left| \int_0^L |v(0,y)|^2 dy \right| \leq C\int_0^L \left(\int_\R |v|^2 dx\right)^\frac{1}{2} \left(\int_\R |\partial_x v|^2 dx\right)^\frac{1}{2} dy
    \end{equation}
    for any $v \in H^1(\Strip_L)$. 
\end{lem}

\begin{proof}
    Let $v \in C^{1}_c(\Strip_L)$ be such that $\supp(v)\subset \R^2$ is compact (as a subset of $\R^2$). Then we have 
    \begin{multline*}
                |v(0,y)|^2 = \int_{-\infty}^0 \partial_x |v(x,y)|^2 dx =  2 \int_{-\infty}^0 Re(\overline{ v(x,y)} \partial_x v(x,y)) dx 
                \\
                \leq 2 \int_{-\infty}^\infty  |v(x,y)| |\partial_x v(x,y)| dx .
    \end{multline*}
    Then \eqref{eqH12Control} follows by integration in $y$. The conclusion holds for any $v\in H^1(\Strip_L)$ by density. 
\end{proof}

\subsection{One-dimensional ground states}\label{sec1dGs}
The one-dimensional counterpart  of the two-dimensional model \eqref{eqDeltaStrip} is the equation
\begin{equation}\label{eq1DEqIntr}
    -\partial_{xx} u + \omega u + \gamma \delta_0 u - |u|^{p-1} u = 0.
\end{equation}
 Here, $\delta_0$ is the Dirac distribution at the origin, namely, $\dual{\delta_0}{ v} = v(0)$ for $v \in H^1(\R)$. 

In the case $\gamma = 0$, the set of solutions of \eqref{eq1DEqIntr} is given by 
\begin{equation*}
    \{e^{i \alpha} \phi_{\omega, 0}(. - z) : \alpha \in [0,2\pi), \, z \in \R \}
\end{equation*}
where the profile $\phi_{\omega,0}$ is explicitly calculated by direct integration of the equation and is given by
\begin{equation}
\label{eq:explicit}
     \phi_{\omega, 0}(x) = \left( \frac{(p+1)\omega}{2} \sech^2 \left( \frac{(p-1) \sqrt{\omega}}{2} |x|   \right)\right)^\frac{1}{p-1}.
 \end{equation}
 Note that we are in the framework of Schr\"odinger equations and therefore the functions that we consider are a priori \emph{complex valued}. 

 When $\gamma \neq 0$, solutions of \eqref{eq1DEqIntr} and their relations to the nonlinear Schr\"odinger dynamics have been thoroughly investigated, from the initial work of Goodman, Holmes and Weinstein \cite{GoHoWe04}, and the stability studies of Fukuizumi and co. \cite{FuJe08,FuOhOz08,CoFuFi08}, up to more recent advanced studies such as the classification of global dynamics of even solutions by Gustafson and Inui \cite{GuIn24} or the construction of a minimal blow-up mass solution by Genoud, Le Coz, and Royer \cite{GeLeRo23}.

Most of the results on solutions to \eqref{eq1DEqIntr} that we are going to use in the present paper have been established in \cite{FuJe08,FuOhOz08,CoFuFi08}. Bounded solutions to \eqref{eq1DEqIntr} exist only when 
\[
\omega > \frac{\gamma^2}{4},
\]
 in which case they can be obtained explicitly by surgery from \eqref{eq:explicit}. Precisely, given $\omega > \frac{\gamma^2}{4}$, there exists a unique positive solution to \eqref{eq1DEqIntr} given by 
    \begin{equation}\label{eqGS1DIntr}
     \phi_{\omega, \gamma}(x) = \left( \frac{(p+1)\omega}{2} \sech^2 \left( \frac{(p-1) \sqrt{\omega}}{2} |x|  - \tanh^{-1}\left( \frac{\gamma}{2 \sqrt{\omega}} \right)  \right)\right)^\frac{1}{p-1}.
  \end{equation}

  The function given by \eqref{eqGS1DIntr} can be characterized as a minimizer of certain variational problems. We define the one-dimensional 
  action, Nehari functional, energy, and mass by (respectively)
\begin{align}
    \label{eqAction1D} 
    S^{1D}_{\omega,\gamma}(u)& = \frac{1}{2} \int_\R |\partial_x u|^2dx  +  \frac{\omega}{2} \int_\R |u|^2 dx+\frac{\gamma}{2} |u(0)|^2 - \frac{1}{p+1} \int_\R |u|^{p+1} dx , \\ 
    \label{eqNehari1D}
    I^{1D}_{\omega,\gamma}(u) &=  \int_\R |\partial_x u|^2dx + \omega \int_\R |u|^2 dx+ \gamma|u(0)|^2- \int_\R |u|^{p+1} dx ,
    \\
    \label{eqEn1D}
    E^{1D}_\gamma(u) &= \frac{1}{2} \int_\R |\partial_x u|^2dx + \frac{\gamma}{2} |u(0)|^2 - \frac{1}{p+1} \int_\R |u|^{p+1} dx , 
    \\ 
    \label{eqMass1D} 
    M^{1D}(u) &= \int_\R |u|^2  dx.
\end{align}
The profile given in \eqref{eqGS1DIntr} was characterized as action ground state in \cite{FuJe08, FuOhOz08,GoHoWe04}.  Moreover, it has also been characterized as an energy ground state by Adami, Noja, and Visciglia \cite{AdNoVi13} when $\gamma<0$. The case $\gamma>0$ for energy ground states has been treated by Boni and Carlone \cite {BoCa23} in the case of the half-line. Their results can be transferred directly to symmetric functions on the line. 
The results can be summarized as follows.

\begin{prop}\label{prp1DIntr}
    Let $\gamma \in \R$ and $\omega > \gamma^2/4$. 
    \begin{itemize}
        \item Let $p>1$. 
    The profile defined in \eqref{eqGS1DIntr} is the unique positive minimizer of 
    \begin{equation}
        \label{eq1DMinIntr}
        \begin{cases}
            s^{1D}_{\omega,\gamma}=\inf\{ S^{1D}_{\omega,\gamma}(u): u \in H^1(\R) \setminus \{0\}, \ I^{1D}_{\omega,\gamma}(u) = 0\} \quad & \mbox{ if } \gamma \leq 0, \\  
            s^{1D}_{\omega,\gamma,sym}=\inf\{ S^{1D}_{\omega,\gamma}(u): u \in H^1_{rad}(\R) \setminus \{0\}, \ I^{1D}_{\omega,\gamma}(u) = 0\} \quad & \mbox{ if } \gamma > 0.
        \end{cases}
    \end{equation}
    \item Let $1<p<5$, $m>0$. There exists $m^*=m^*(\gamma)$, with $m^*(\gamma)=0$ if $\gamma<0$ and $m^*(\gamma)>0$ if $\gamma>0$, such that the following hold. Assume that $m> m^*$.  Then there exists a unique $\omega(m) > \gamma^2/4$ such that the function $\phi_{\omega(m),\gamma}$ defined in \eqref{eqGS1DIntr} is the unique real-valued and positive minimizer of
\begin{equation}
  \label{eqEnMin1D} 
   \begin{cases}
 e^{1D}_{m,\gamma} = \inf\{E^{1D}_\gamma(u)  : u \in H^1(\R),\, M^{1D}(u) = m \}&\text{ if }\gamma\leq0,\\
 e^{1D}_{m,\gamma,sym} = \inf\{E^{1D}_\gamma(u)  : u \in H^1_{rad}(\R),\, M^{1D}(u) = m \}&\text{ if }\gamma>0.\\
      \end{cases}
\end{equation}
    If $m\leq m^*$, then the problems do not admit a minimizer.
    \end{itemize}
\end{prop}

Notice that in the repulsive case $\gamma >0$, symmetry with respect to the origin is required. This originates from the fact that a minimizing sequence may not be compact and may exhibit a run-away behavior at infinity on one side of the line, as shown by Fukuizumi and Jeanjean \cite{FuJe08}. 

Notice also that, even in the symmetric case, minimization of the energy at fixed mass can fail, as it becomes energetically favorable for small masses to divide the sequence into two parts traveling away from the origin.

Using \eqref{eqGS1DIntr}, we can obtain an explicit relation between $\omega$ and $M^{1D}(\phi_{\omega,\gamma})$. We have
\begin{equation}\label{eqMtoOmega1D}
\begin{aligned}
    M^{1D}(\phi_{\omega,\gamma}) & = \left( \frac{(p+1) \omega}{2}\right)^{\frac{2}{p-1}} \int_\R \sech^\frac{4}{p-1} \left( \frac{(p-1)\sqrt{\omega}}{2}|x| - \tanh^{-1} \left( \frac{\gamma}{2\sqrt{\omega}}\right)\right) \, dx  \\
    & = \left( \frac{(p+1) \omega}{2}\right)^{\frac{2}{p-1}} \frac{4}{(p-1)\sqrt{\omega}} \int_{\tanh^{-1} \left( \frac{\gamma}{2\sqrt{\omega}}\right)}^\infty \sech^\frac{4}{p-1} \left( x\right) \, dx \\
    & =: Q(\omega,\gamma) \omega^{\frac{5-p}{2(p-1)}},
    \end{aligned}
\end{equation} 
where 
 \begin{equation}\label{eqCOmega}
     Q(\omega,\gamma) = \left( \frac{(p+1)}{2}\right)^{\frac{2}{p-1}}  \frac{4}{(p-1)} \int_{\tanh^{-1} \left( \frac{\gamma}{2\sqrt{\omega}}\right)}^\infty \sech^\frac{4}{p-1} \left( x\right) \, dx
 \end{equation}
 is well defined and uniformly bounded for $(\omega,\gamma) \in [\frac{\gamma^2}{4}, \infty) \times \R$ by 
 \begin{equation*}
     Q(\omega,\gamma)  \leq \left( \frac{(p+1)}{2}\right)^{\frac{2}{p-1}}  \frac{4}{(p-1)} \int_\R \sech^\frac{4}{p-1} \left( x\right) \, dx < \infty.
 \end{equation*}

 We use the following identity.
 \begin{lem}
     Any solution $u \in H^1(\R)$ to \eqref{eq1DEqIntr} satisfies 
     \begin{align} \label{eqEnMassZero}
        2 (p+3)E^{1D}_\gamma(u) = - (5-p) \omega M^{1D}(u) + (p-1)\gamma |u(0)|^2.
     \end{align}
 \end{lem}

 \begin{proof}
     The unique solutions to \eqref{eq1DEqIntr} are given by 
     \begin{equation*}
         \left\{ e^{i \alpha} \phi_{\omega, \gamma} : \alpha \in [0,2\pi) \right\}
     \end{equation*}
     where $\phi_{\omega,\gamma}$ is defined in \eqref{eqGS1DIntr}. In particular, the following duality products are well-defined. 
     Taking the scalar product of equation \eqref{eq1DEqIntr} with $u$ and $x \partial_x u$, we get the following two identities
     \begin{equation} \label{eqPoho1D}
         \begin{aligned}
              &  \gamma |u(0)|^2 + \int_\R |\partial_x u|^2 dx+ \omega \int_\R |u|^2 dx- \int_\R |u|^{p+1}  dx  = 0,\\
         & \int_\R |\partial_x u|^2dx - \omega \int_\R |u|^2dx + \frac{2}{p+1}\int_\R  |u|^{p+1}  dx = 0.
         \end{aligned}
     \end{equation}
     Multiplying the first equation by $\frac{4}{p+3}$, the second by $\frac{p-1}{p+3}$ and summing, we obtain
     \[
     \gamma \frac{4}{p+3} |u(0)|^2 + \int_\R |\partial_x u |^2dx - \frac{2}{p+1} \int_\R | u |^{p+1}dx + \omega \frac{5-p}{p+3} \int_\R |u|^2  dx  = 0.
     \]
     The result follows. 
 \end{proof}
 As a direct consequence, we have the following. 
 \begin{lem}
     For any $\phi_{\omega,\gamma}$ in \eqref{eqGS1DIntr},  we have 
      \begin{gather}
         \label{eqGS1Din0}
         \phi_{\omega,\gamma}(0) = \left( \frac{p+1}{2} \left(\omega - \frac{\gamma^2}{4} \right)\right)^\frac{1}{p-1},
\\
         \label{eqEnPhiOmega}
         2(p+3)E^{1D}_\gamma (\phi_{\omega,\gamma}) = - (5-p) Q(\omega,\gamma) \omega^\frac{p+3}{2p-2} + (p-1) \gamma \left(\frac{p+1}{2} \left( \omega - \frac{\gamma^2}{4}\right) \right)^\frac{2}{p-1}.
     \end{gather}
 \end{lem}

 \begin{proof}
 The relation \eqref{eqGS1Din0} follows from the algebraic identity $\sech(\tanh^{-1}(x)) = \sqrt{1-x^2}.$
 Combined with \eqref{eqMtoOmega1D} and \eqref{eqEnMassZero}, this implies \eqref{eqEnPhiOmega}.     
 \end{proof}

Finally, the sign of 
$ \partial_\omega M^{1D}(\phi_{\omega,\gamma})$ has been previously determined by direct calculations (see  \cite{FuJe08,FuOhOz08,CoFuFi08}).
 \begin{lem}\label{lemStabil}
     If $\gamma < 0$, then the following holds.
     \begin{enumerate}
         \item If $1 < p \leq 5$, then $\partial_\omega M^{1D}(\phi_{\omega,\gamma}) >0$.
         \item If $p >5$, then there exists $\omega_1>\gamma^2/4$ such that $\partial_\omega M^{1D}(\phi_{\omega,\gamma}) >0$ for $\gamma^2/4<\omega < \omega_1$ and  $\partial_\omega M^{1D}(\phi_{\omega,\gamma}) < 0$ for $\omega > \omega_1$. 
     \end{enumerate}
      If $\gamma > 0$, then the following holds.
     \begin{enumerate}
         \item If $1 < p \leq 3$, then $\partial_\omega M^{1D}(\phi_{\omega,\gamma}) >0$.
         \item If $3 < p < 5$, then there exists $\omega_2>\gamma^2/4$ such that $\partial_\omega M^{1D}(\phi_{\omega,\gamma}) >0$ for $\omega > \omega_2$ and  $\partial_\omega M^{1D}(\phi_{\omega,\gamma}) < 0$ for $\gamma^2/4<\omega < \omega_2$.
         \item If $p >5$, then $\partial_\omega M^{1D}(\phi_{\omega,\gamma}) < 0$.
     \end{enumerate}
 \end{lem}

\section{Existence of action minimizers}\label{secActionMin}

In this section, we prove that solutions to \eqref{eqDeltaStrip} exist as minimizers of the action over the Nehari manifold.

The existence of the minimum depends on the choices of $\omega$ and $\gamma$. It is connected with the coercivity of the following operator 
\begin{equation}
        \label{eqLGamma}
         \mathcal {L}_{\omega,\gamma} = - \Delta + \gamma \deltastrip  + \omega 
    \end{equation} 
which we will now analyse in further details.

The following result follows from direct computations.
\begin{lem}\label{lemTestComput}
    Let $\gamma < 0$, $f_\gamma(x,y) = \sqrt{\frac{-\gamma}{2L}} e^{\frac{\gamma|x|}{2}}$. Then we have 
    \begin{equation}
    \begin{aligned}
        \label{eqTestFunMass} &\| f_\gamma \|_{L^2(\Strip_L)} = 1, \quad  \| \nabla f_\gamma \|_{L^2(\Strip_L)}^2 = \frac{\gamma^2}{4}, \quad \gamma \int_0^L |f_\gamma(0,y)|^2 dy = - \frac{\gamma^2}{2}, \\ 
         &\| f_\gamma \|_{L^{p+1}(\Strip_L)}^{p+1} = \frac{2^{\frac{3-p}{2}}}{(p+1)L^\frac{p-1}{2}} (-\gamma)^{\frac{p-1}{2}}.
    \end{aligned}
    \end{equation}
\end{lem}
\begin{remark}
    The function $f(x) = \sqrt{\frac{-\gamma}{2}} e^{\frac{\gamma|x|}{2}}$ is the normalized eigenfunction for the smallest eigenvalue  $-\gamma^2/4$ of the one-dimensional operator $-\partial_{xx}+\gamma \delta$.
\end{remark}
As a consequence of Lemma \ref{lemTestComput}, we obtain the following. 
\begin{lem}\label{lemLambdaGamma}
Suppose $\gamma < 0$. For any $L >0$ and $h >0$, define
\begin{equation}\label{eqLambdaGamma}
   \lambda_{\gamma,h}  := \inf_{ \{ u \in H^1(\Strip_L) : \| u \|_{L^2(\Strip_L)} = 1 \} }\left\{\int_0^L \left( \int_\R |\partial_x u(x,y)|^2 + h |\partial_y u(x,y)|^2 dx\right) + \gamma |u(0,y)|^2 dy \right\}.
\end{equation}
Then 
\begin{equation}
   \lambda_{\gamma,h}  = -\frac{\gamma^2}{4}.
\end{equation}
\end{lem}

\begin{proof}
    On one hand, from  \eqref{eqTestFunMass} we get
    \[
    \lambda_{\gamma,h} \leq \| \nabla f_\gamma \|_{L^2(\Strip_L)}^2 + \gamma\int_0^L |f_\gamma(0,y)|^2 dy = - \frac{\gamma^2}{4}. 
    \]
    On the other hand, we have 
    \begin{align*}
        \lambda_{\gamma,h} \geq \mu_\gamma = \inf \left\{\int_0^L \left(\int_\R |\partial_x u (x,y)|^2 dx + \gamma |u(0,y)|^2 \right)dy\, :u \in H^1(\Strip_L),\, \int_0^L  M^{1D}(u(y)) \, dy  =1 \right\}.
    \end{align*}
    Let $u\in H^1(\Strip_L)$ be such that $\norm{u}_{L^2(\Strip_L)}=1$. Define $\theta:[0,L]\to [0,\infty]$ by 
    \[
    \theta(y)=\int_\R |u(x,y)|^2dx.
    \]
    Then $\int_0^L|\theta(y)|^2dy=1$. In particular, $|\theta(y)|<\infty$ for almost every $y\in[0,L]$. Let $y\in [0,L]$ be such that $0\leq \theta(y)<\infty$. 
    Consider the one-dimensional problem 
    \[
    m_ {\theta(y)} = \inf \left\{ \int_\R | \partial_x g |^2 dx  + \gamma |g(0)|^2 \, :\, g\in H^1(\R), \,   M^{1D}( g )   =  {\theta(y)}^2 \right\}. 
    \] 
    The minimum $m_ {\theta(y)}$ is achieved in $ {\theta(y)} f_\gamma$, and it is 
    \[
    m_ {\theta(y)} = -  {\theta(y)}^2 \frac{\gamma^2}{4} = -  M^{1D}(  {\theta(y)} f_\gamma)\frac{\gamma^2}{4}.
    \] 
    This implies that 
    \[
    \int_0^L \left(\int_\R |\partial_x u (x,y)|^2 dx + \gamma |u(0,y)|^2 \right)dy
    \geq -  \int_0^LM^{1D}(  {\theta(y)} f_\gamma)\frac{\gamma^2}{4}dy,
    \]
    and consequently
    \[ 
    \mu_\gamma \geq \inf \left\{ - \frac{  \gamma^2}{4} \int_0^L  M^{1D}(  {\theta(y)} f_\gamma) \, dy\, :\, \int_0^L  M^{1D}(  {\theta(y)} f_\gamma) \, dy  =1 \right\} = - \frac{\gamma^2}{4}.
    \]
    This concludes the proof.
\end{proof}

\begin{remark}
    When $\gamma >0$, then we notice that $\lambda_{\gamma,h} = 0$ in \eqref{eqLambdaGamma}.  Indeed, for any $u \in H^1(\Strip)$ such that $\| u \|_{L^2(\Strip_L)} = 1$,  define $u_\lambda$ by $u_\lambda(x,y) = \lambda^\frac{1}{2} u(\lambda x, y)$. Then $\| u_\lambda\|_{L^2(\Strip_L)} = \| u \|_{L^2(\Strip_L)}$ while $ \| \nabla u_\lambda \|_{L^2(\Strip_L)}^2 = \lambda^2 \| \partial_x u \|_{L^2(\Strip_L)} + \| \partial_y u \|_{L^2(\Strip_L)}^2$, and 
    \[
    \int_0^L |u_\lambda (0,y)|^2 dy = \lambda \int_0^L |u(0,y)|^2 dy.
    \] 
    So, by taking $u$ not depending on $y$ and taking the limit $\lambda \to 0$, we obtain that $\lambda_{\gamma,h} = 0$.
\end{remark}

\begin{lem}
    Let $\gamma \geq 0$ and $\omega > 0$  or $\gamma <0$ and $\omega > \frac{\gamma^2}{4}$. Then there exists $K(\omega, \gamma) = K >0$ 
    such that 
    \begin{equation} \label{eqCoercH1}
        \dual{ \mathcal {L}_{\omega,\gamma} v}{v} \geq K \| v \|_{H^1(\Strip_L)}^2
    \end{equation}
    for any $v \in H^1(\Strip_L) $. 
\end{lem}

\begin{proof}
For $\gamma \geq 0$, $\omega >0$, the statement is immediate. For $\gamma <0$, Lemma \ref{lemLambdaGamma} implies that
    \begin{equation}\label{eqCoercL2}
        \dual{ \mathcal {L}_{\omega,\gamma} v}{v} \geq (\omega - \frac{\gamma^2}{4}) \| v \|_{L^2(\Strip_L)}^2.
    \end{equation}  
     Now suppose \eqref{eqCoercH1} is false. Then there exists a sequence $(v_n) \subset H^1(\Strip_L)$ such that $\| v_n \|_{H^1(\Strip_L)} \to  \epsilon > 0$ and $\dual{\mathcal {L}_{\omega,\gamma} v_n}{v_n} \to 0$.  Thus, from \eqref{eqCoercL2} we have $\| v_n \|_{L^2(\Strip_L)} \to 0$, and 
     \begin{equation} \label{eqContr1}
        0 = \lim_{n \to \infty}  \dual{\mathcal {L}_{\omega,\gamma} v_n}{ v_n} = \eps + \lim_{n \to \infty} \gamma \int_0^L |v_n (0,y)|^2 dy .
    \end{equation} 
    From \eqref{eqH12Control} and Young inequality, we have that for any $\eta >0$, 
    \[  
    \gamma \left| \int_0^L |v_n(0,y)|^2 dy \right| \lesssim \int_0^L \| v_n \|_{L^2_x(\R)} \| \partial_x v_n \|_{L^2_x(\R)} dy  \lesssim \frac{1}{\eta} \| v_n \|_{L^2(\Strip_L)}^2 + \eta\| \nabla v_n \|_{L^2(\Strip_L)}^2 \to \eta \eps 
    \]
    as $n \to \infty$. This contradicts \eqref{eqContr1} for $\eta$ small enough. 
\end{proof}

With the help of the function $f_\gamma$, we may also prove the following non-existence result.
\begin{prop}
\label{prop:non-existence}
    Let $\omega\in\mathbb R$ and $\gamma<0$. Assume that $\omega\leq \frac {\gamma^2}4$. Then there does not exist a non-trivial non-negative solution to \eqref{eqDeltaStrip}.
\end{prop}
\begin{proof}
Let $u$ be a non-trivial non-negative solution of \eqref{eqDeltaStrip}. 
  Multiplying \eqref{eqDeltaStrip} by $f_\gamma$ and using the self-adjointness of $\mathcal L_{\omega,\gamma}$, we obtain the following identity:
    \[
    \dual{u}{L_{\omega,\gamma}f_\gamma}-\int_{\Strip_L}|u|^{p-1}uf_\gamma dxdy=0
    \]
    We now use the fact that $\mathcal L_{\omega,\gamma} f_\gamma = \left(\omega-\frac{\gamma^2}4\right)f_\gamma$ to get
    \[
    \left(\omega-\frac{\gamma^2}4\right)
     \int_{\Strip_L}uf_\gamma dxdy=\int_{\Strip_L}|u|^{p-1}uf_\gamma dxdy.
     \]
     As $f_\gamma>0$ and $u\geq 0$, $u\neq 0$, this identity can hold only when $\omega>\frac{\gamma^2}4$.
\end{proof}

We now discuss the minimization problems.

It is easily seen that a constraint is necessary to obtain a minimizer. Indeed, the unconstrained action does not admit a lower bound: for any $u \in H^1(\Strip_L)$, given $\lambda \in \R$, we have
\begin{equation*}
    S_{\omega, \gamma}(\lambda u) = \frac{\lambda^2}{2} \left( \| \nabla u \|_{L^2(\Strip_L)}^2 + \omega \| u \|_{L^2(\Strip_L)}^2 + \gamma \int_0^L |u(0,y)|^2 dy \right) - \frac{\lambda^{p+1}}{p+1} \| u \|_{L^{p+1}(\Strip_L)}^{p+1}. 
\end{equation*}
By taking $\lambda \to \infty$, as $p > 1$, one see that $S_{\omega,\gamma}(\lambda u) \to - \infty$. 

The action restricted to the Nehari manifold is bounded from below. In fact,
the minimization problem \eqref{eqActNehIntr} can be rewritten as minimization of the potential part of the action over the Nehari manifold:
\begin{equation}\label{eqDFirstDef}
    s_{\omega,\gamma}=%
    \frac{p-1}{2(p+1)} \inf\left\{\| u \|_{L^{p+1}(\Strip_L)}^{p+1} : u \in H^1(\Strip_L) \setminus\{0\}, \, I_{\omega,\gamma}(u)= 0 \right\}.
\end{equation} 
Remark that the Nehari constraint is a natural constraint for  \eqref{eqDeltaStrip}, i.e. all the solutions to \eqref{eqDeltaStrip} verify it. In particular, a minimizer of $s_{\omega,\gamma}$ is a critical point of the \emph{unconstrained} function $S_{\omega,\gamma}$, i.e. the associated Lagrange multiplier is $0$. Indeed, a constrained minimizer $\psi$ satisfies the Euler-Lagrange equation 
\[ 
S_{\omega, \gamma}' (\psi) = \Lambda I_{\omega, \gamma}'(\psi)  
\]
for some Lagrange multiplier $\Lambda \in \R$. This implies that 
\[ 
\dual{S_{\omega, \gamma}' (\psi)}{ \psi} 
= I_{\omega, \gamma}(\psi) = 0 
= \Lambda \dual{I_{\omega, \gamma}'(\psi)}{ \psi} 
= \Lambda (p-1) \| \psi \|_{L^{p+1}(\Strip_L)}^{p+1},
\]
which leads to the desired conclusion. 
 
We now show that the variational problem  \eqref{eqDFirstDef} is equivalent to 
\begin{equation}\label{eqDNonRad}
   \frac{p-1}{2(p+1)} \inf\left\{\| u \|_{L^{p+1}(\Strip_L)}^{p+1}, \, u \in H^1(\Strip_L) \setminus\{0\}, \, I_{\omega,\gamma}(u)\leq 0 \right\},
\end{equation}
and that a minimizer can be chosen to be real-valued and non-negative.

\begin{lem}
\label{lem:equiv-min-prob1}
    Assume that $\psi \in H^1(\Strip_L)$ is a minimizer for \eqref{eqDNonRad}. Then $I_{\omega, \gamma}(\psi) = 0$, $|\psi|$ is also a minimizer, and problem \eqref{eqDFirstDef} is equivalent to problem \eqref{eqDNonRad}. 
\end{lem}

\begin{proof}
Assume, by contradiction, that $\psi$ is a minimizer for $\eqref{eqDNonRad}$ and $I_{\omega,\gamma}(\psi) < 0$.  Let $t: H^1(\Strip_L) \to \R$ be defined by 
\begin{equation}\label{eqTUnDefinit}
    t(\psi) = \left(  \frac{\| \nabla \psi \|_{L^2(\Strip_L)}^2 + \omega \| u \|_{L^2(\Strip_L)}^2  + \gamma \int_0^L | \psi(0,y)|^2 dy }{\| \psi \|_{L^{p+1}(\Strip_L)}^{p+1}} \right)^{\frac{1}{p-1}} =  \left( 1 + \frac{I_{\omega,\gamma}(\psi)}{\| \psi \|_{L^{p+1}(\Strip_L)}^{p+1}} \right)^{\frac{1}{p-1}}.
\end{equation}
We have that $I_{\omega,\gamma}(t(\psi) \psi) = 0$, $t(\psi) < 1$  and 
\[ 
\frac{p-1}{2(p+1)} \| t(\psi) \psi \|_{L^{p+1}(\Strip_L)}^{p+1} < \frac{p-1}{2(p+1)} \|  \psi \|_{L^{p+1}(\Strip_L)}^{p+1}%
\] 
contradicting the definition of $\psi$ as a minimizer for \eqref{eqDNonRad}.

We show that $|\psi|$ is also a minimizer.  Indeed, since $\| \nabla |\psi| \|_{L^2(\Strip_L)} \leq \| \nabla \psi \|_{L^2(\Strip_L)}$, $|\psi| \geq 0$ is a real minimizer for \eqref{eqDNonRad} and $I_{\omega,\gamma}(|\psi|) \leq I_{\omega,\gamma}(\psi)=0$. As before, if  $I_{\omega,\gamma}(|\psi|) < I_{\omega,\gamma}(\psi)$, then $t(|\psi|) |\psi| $ satisfies
\[ 
\frac{p-1}{2(p+1)} \| t(|\psi|)| \psi| \|_{L^{p+1}(\Strip_L)}^{p+1} < \frac{p-1}{2(p+1)} \| \psi \|_{L^{p+1}(\Strip_L)}^{p+1} %
\]
which is a contradiction.
 Therefore $I_{\omega,\gamma}(|\psi|) = 0$, $\| \nabla |\psi| \|_{L^2(\Strip_L)} = \| \nabla \psi \|_{L^2(\Strip_L)}$, and $|\psi|$ is also a minimizer for \eqref{eqDNonRad}.
\end{proof}

The rest of this section is separated into two cases, 
 $\gamma >0$ and $\gamma < 0$. This is due to the different approaches employed to show the existence of minimizers. Indeed,  $\gamma >0$ yields additional difficulties due to the possibility of minimizing sequences escaping to infinity on one side of the strip, see Remark \ref{remNotExis} below.

\subsection[The attractive case]{Action minimizers in the attractive case}
In this section, we show Theorem \ref{thm:action-attract}. 

As we often make use of the existence of a minimizer for $s_{\omega, 0}$, we show the following result,  which does not seem to have been previously established in the literature. 

\begin{prop}\label{prpGamma0}
    For any $\omega >0$, $1 < p$, there exists a non-negative minimizer for the problem $s_{\omega, 0}$.
\end{prop}

\begin{proof}

Let $(u_n)$ be a minimizing sequence for  $s_{\omega, 0}$. Then $(|u_n|)$ is also a minimizing sequence, thus we suppose $u_n$ to be non-negative.  Then 
\begin{equation}\label{eqH1LPEq}
    \|  \nabla u_n \|_{L^2(\Strip_L)}^{2} + \omega \| u_n \|_{L^2(\Strip_L)}^{2}  \leq \| u_n \|_{L^{p+1}(\Strip_L)}^{p+1} \to s_{\omega, 0} < \infty, 
\end{equation}
thus $(u_n)$ is bounded in $H^1(\Strip_L)$.  Then, up to a subsequence, $u_n \rightharpoonup u$ in $H^1(\Strip_L)$  and $u_n \to u$ $a.e.$.  If by contradiction $u_n \to 0$ in $L^{p+1}(\Strip_L)$, then from \eqref{eqH1LPEq} $\| u_n \|_{H^1(\Strip_L)} \to 0$ which is absurd because, by Gagliardo-Nirenberg inequality 
\begin{equation*}
    \min(\omega, 1) \| u_n \|_{H^1(\Strip_L)}^2 \leq  \|  \nabla u_n \|_{L^2(\Strip_L)}^{2} + \omega \| u_n \|_{L^2(\Strip_L)}^{2}  = \| u_n \|_{L^{p+1}(\Strip_L)}^{p+1} \leq \| u_n \|_{H^1(\Strip_L)}^{p+1}  
\end{equation*}
thus 
\begin{equation}
    0 < \min(\omega, 1) \leq \| u_n \|_{H^1(\Strip_L)}^{p-1}.
\end{equation}
So $u_n \not \to 0$ in $L^{p+1}(\Strip_L)$ and thus from Lions' Lemma \cite[Lemma 1.21]{Wi96} it follows that 
\begin{equation*}
   \int_{B(z_n,\frac{L}{4})} u_n^2 dx dy \geq \eps
\end{equation*}
for some $\eps >0$, $z_n \in \Strip_{L}$ (up to a subsequence). By translating  $u_n$ with respect to the variable $x$ and passing to a subsequence if necessary, we assume that $(z_n)$ is bounded and $u_n \rightharpoonup u \not \equiv 0 $. We conclude that $I_{\omega,0}(u) \leq 0$ and $s_{\omega,0} = S_{\omega,0}(u)$ by the Brezis-Lieb lemma (see \cite{BrLi83}) (see the second part of the proof of Lemma \ref{lemGammaNeg} for more details on this argument in the generic case $\gamma \in \R$). 
\end{proof}

We show the following preliminary result. 

\begin{lem}\label{lem:compare_s}
    Let $\gamma < 0$ and $\omega > \gamma^2/4$. Then 
    \begin{equation}
        \label{eqDgammaD0}
        s_{\omega,\gamma} < s_{\omega,0}.
    \end{equation}
\end{lem}

\begin{proof}
    Let $\psi \not\equiv 0$ be a minimizer for $s_{\omega,0}$.  Up to a space translation, we may assume that $\psi(0,y) \not\equiv  0$. Thus, we have that \[I_{\omega,\gamma}(\psi) = I_{\omega,  0} (\psi) + \gamma \int_0^L |\psi(0,y)|^2 dy < 0.\] This implies that  $t(\psi) < 1$ and $I_{\omega, \gamma}(t(\psi)\psi) = 0$ where $t(\psi)$ is defined in \eqref{eqTUnDefinit}. As a consequence, we obtain 
    \[
    s_{\omega,\gamma} \leq \frac{p-1}{2(p+1)} \| t(\psi) \psi\|_{L^{p+1}(\Strip_L)}^{p+1} < s_{\omega,0}.
    \] 
    This shows the desired statement.
\end{proof}

\begin{lem}
\label{lem:whatsyourname}
    For any $\gamma < 0$ and $\omega > \gamma^2/4$, we have $s_{\omega,\gamma} >0$.
\end{lem}

\begin{proof}
    From \eqref{eqCoercH1} and Gagliardo-Nirenberg inequality, for any $u \in H^1(\Strip_L)\setminus \{0\}$ such that $I_{\omega,\gamma}(u) = 0$, there exist $c_1,C_2>0$ such that
    \[
    c_1\| u \|_{H^1(\Strip_L)}^2 \leq \dual{ \mathcal {L}_{\omega,\gamma} u}{ u} = \| u \|_{L^{p+1}(\Strip_L)}^{p+1} \leq C_2 \| u \|_{H^1(\Strip_L)}^{p+1}.
    \]  
    This yields the uniform bound $  \| u \|_{H^1(\Strip_L)}^{p-1}\geq c_1/C_2>0$ and $ \dual{ \mathcal {L}_{\omega,\gamma} u}{ u} \geq c_1(c_1/C_2)^{2/(p-1)}>0$. Since $I_{\omega,\gamma}(u) = 0$, we also obtain that $ \| u \|_{L^{p+1}(\Strip_L)}^{p+1}\geq c_1(c_1/C_2)^{2/(p-1)}>0$.
\end{proof}

We now show the main proposition of this section.
\begin{lem}\label{lemGammaNeg}
    Let $\gamma < 0$ and $\omega > \gamma^2/4$. Then \eqref{eqDNonRad} admits a nontrivial minimum. 
\end{lem}

\begin{proof}
        Let $(u_n) \subset H^1 (\Strip_L) \setminus \{0\} $ be a minimizing sequence for $s_{\omega,\gamma}$. Remark that by construction,  $\| u_n \|_{L^{p+1}(\Strip_L)} $ is bounded. Without loss of generality, we assume that $ I_{\omega,\gamma}(u_n) = 0$. This implies that $(u_n)$ is bounded in $ H^1 (\Strip_L)$.
    Thus, up to a subsequence, there exists $u \in H^1(\Strip_L)$ such that $u_n \rightharpoonup u$ in $H^1(\Strip_L)$.
     
    Next, we prove that $u \not\equiv 0$. Assume by contradiction that $ u \equiv  0$. Let $\tilde t: H^1(\Strip_L) \to \R$ be defined by
    \[  
    \tilde t (\psi) = \left( 1 + \frac{I_{\omega,\gamma}(\psi) - \gamma \int_0^L | \psi(0,y)|^2 dy }{\| \psi \|_{L^{p+1}(\Strip_L)}^{p+1}}\right)^{\frac{1}{p-1}}. 
    \]
    We know that 
    \begin{equation}\label{eqIUnZero}
        I_{\omega,\gamma}(u_n) = 0.
    \end{equation} 
    Moreover, from Lemma \ref{lemTraceComp}, we obtain that, up to a subsequence,   
    \[
    \int_0^L | u_n(0,y)|^2 dy \to 0.
    \]
    Consequently $\tilde t(u_n) \to 1$ as $n \to \infty$. Thus, on the one hand, we obtain that 
    \[
    \frac{p-1}{2(p+1)} \| \tilde t(u_n)  u_n \|_{L^{p+1}(\Strip_L)}^{p+1} \to s_{\omega,\gamma}
    \] 
    and on the other hand 
    \[
    s_{\omega,\gamma} < s_{\omega,0} \leq  \frac{p-1}{2(p+1)} \| \tilde t(u_n)  u_n \|_{L^{p+1}(\Strip_L)}^{p+1} \to s_{\omega,\gamma},
    \] 
    where the first inequality is \eqref{eqDgammaD0} while the second follows from $I_{\omega,  0}(\tilde t(u_n)  u_n) = 0$. This is a contradiction and implies that we cannot have $ u \equiv  0$. 
    
    Now we prove that $I_{\omega,\gamma}(u) \leq 0$. From Lemma \ref{lemTraceComp} and the Brezis-Lieb lemma (see \cite{BrLi83}) we obtain that
    \[ 
    I_{\omega,\gamma}(u_n) - I_{\omega,\gamma}(u_n - u) - I_{\omega,\gamma}(u) \to 0
    \] 
    as $n\to \infty$. Suppose that $I_{\omega,\gamma}(u) > 0$.  Then, from \eqref{eqIUnZero} we get that 
    \[ 
    \lim_{n \to \infty} I_{\omega,\gamma}(u_n - u) = \lim_{n \to \infty} I_{\omega,\gamma}(u_n) - I_{\omega,\gamma}(u) = -I_{\omega,\gamma}(u)< 0.
    \] 
    Thus, there exists $N \in\mathbb N$ such that for any $n > N$, we have $ I_{\omega,\gamma}(u_n - u) <0.$ This implies that, for $n >N$, \begin{equation}
        \label{eqContradict1}
        s_{\omega,\gamma} \leq \frac{p-1}{2(p+1)} \| u_n - u \|_{L^{p+1}(\Strip_L)}^{p+1}.
    \end{equation} 
    Since $u \not\equiv 0$, from the Brezis-Lieb lemma, we obtain 
    \[ 
    \lim_{n \to \infty} \frac{p-1}{2(p+1)} \| u_n - u \|_{L^{p+1}(\Strip_L)}^{p+1} = \lim_{n \to \infty} \frac{p-1}{2(p+1)} \| u_n \|_{L^{p+1}(\Strip_L)}^{p+1} -  \frac{p-1}{2(p+1)} \| u \|_{L^{p+1}(\Strip_L)}^{p+1} < s_{\omega,\gamma}
    \] 
    which contradicts \eqref{eqContradict1}. 
    Consequently we have that $I_{\omega,\gamma}(u) \leq 0$ and $s_{\omega,\gamma} \leq S_{\omega,\gamma}(u)$. 
    On the other hand, from the weak lower semicontinuity, we get 
    \[
    \frac{p-1}{2(p+1)} \| u \|_{L^{p+1}(\Strip_L)}^{p+1} \leq \lim_{n \to \infty} \frac{p-1}{2(p+1)} \| u_n \|_{L^{p+1}(\Strip_L)}^{p+1} = s_{\omega,\gamma},
    \] 
    implying that 
    \[
    \frac{p-1}{2(p+1)} \| u \|_{L^{p+1}(\Strip_L)}^{p+1} = s_{\omega,\gamma}.
    \]
    Hence $u$ is a non-trivial minimizer for \eqref{eqDNonRad}, which concludes the proof. 
\end{proof}

\begin{proof}[Proof of Theorem \ref{thm:action-attract}]
    The theorem is a direct consequence of Lemma \ref{lem:equiv-min-prob1} and Lemma \ref{lemGammaNeg}.
\end{proof}

\subsection[The repulsive case]{Action minimizers in the repulsive case}
\label{secActionMinRepuls}

\begin{remark}\label{remNotExis}
In the repulsive case, a minimizer does not exist in $H^1(\Strip_L)$. Indeed, sequences of functions that achieve the minimum escape to infinity. More precisely,
    let $\gamma >0$ and suppose that $\psi$ is a real-valued minimizer for \eqref{eqDNonRad}. Since $\psi$ satisfies equation \eqref{eqDeltaStrip}, we conclude that $\psi \in C(\overline{\Strip_L})$, and it decays exponentially to zero as $|x| \to \infty$ (see Appendix \ref{sec:properties}). Thus, we may choose  $z\in\R$ large enough such that $I(\tau_z \psi) < I(\psi)$, where $\tau_z \psi (x) = \psi(x - z)$. This implies that $t(\tau_z \psi) < 1$ where $t(.)$ is defined in \eqref{eqTUnDefinit}, $I(t(\tau_z \psi) \tau_z \psi) = 0$ and 
    \[ 
    s_{\omega,\gamma} \leq \frac{p-1}{2(p+1)} \| s \tau_z \psi\|_{L^{p+1}(\Strip_L)}^{p+1} < \frac{p-1}{2(p+1)} \| \psi\|_{L^{p+1}(\Strip_L)}^{p+1}  = s_{\omega,\gamma}.
    \]
    This is in contradiction with the minimizing nature of $\psi$.
\end{remark}

To avoid the issue explained in Remark \ref{remNotExis}, we develop our analysis in $ H^1_{sym}(\Strip_L)$, the subset of $H^1(\Strip_L)$ given by symmetrical with respect to $x$ functions, defined in \eqref{eqH1Symmetric}. 

We consider the following variational problem
\begin{equation}
    \label{eqDRadial}
    s_{\omega,\gamma,sym} = \inf \left\{\frac{p-1}{2(p+1)} \| v \|_{L^{p+1}(\Strip_L)}^{p+1}: v \in H^1_{sym}(\Strip_L) \setminus \{0\} : I_{\omega,\gamma}(v) \leq 0 \right\}.
\end{equation}

\begin{lem}
    Let $\gamma > 0$ and $\omega > 0$.    
        Assume that $\psi \in H^1(\Strip_L)$ is a minimizer for \eqref{eqDRadial}. Then $I_{\omega, \gamma}(\psi) = 0$, $|\psi|$ is also a minimizer, and problem \eqref{eqDRadial} is equivalent to problem \eqref{eqActNehRadIntr}.  
\end{lem}

\begin{proof}
The proof follows the same line of arguments as the one of Lemma \ref{lem:equiv-min-prob1}.
\end{proof}

We have the following relation between the Nehari minimum $s_{\omega,\gamma}$ and the symmetric Nehari minimum $s_{\omega,\gamma,sym}$.

\begin{lem}
    Let $\gamma > 0$ and $\omega > 0$. Then $ s_{\omega,\gamma} \leq s_{\omega,0} \leq s_{\omega,\gamma,sym}$.
\end{lem}

\begin{proof}
    On one hand, assume that $\psi$ is a minimizer for $s_{\omega,\gamma,sym}$, that is 
    \[ 
    s_{\omega,\gamma,sym} = \frac{p-1}{2(p+1)}\| \psi\|_{L^{p+1}(\Strip_L)}^{p+1}
    \]
    and $I_{\omega,\gamma}(\psi) = 0$. 
    Therefore 
    \[ 
    I_{\omega,  0}(\psi) = - \gamma \int_0^L | \psi(0,y)|^2 dy \leq 0.
    \]
    Let 
    \[ 
    s = \left(1 + \frac{I_{\omega, 0}(\psi)}{\| \psi\|_{L^{p+1}(\Strip_L)}^{p+1}}\right)^{\frac{1}{p-1}}.
    \]
    Then $s \leq 1$, $I_{\omega, 0}(s\psi)=0$ and 
    \[ 
    s_{\omega,0} \leq \frac{p-1}{2(p+1)} \| s \psi \|_{L^{p+1}(\Strip_L)}^{p+1} \leq s_{\omega,\gamma,sym}.
    \]
    
     On the other hand, let $\phi$ be the minimizer for $s_{\omega,0}$. 
     From the classical elliptic theory (see Appendix \ref{sec:properties} for details), we know that $\phi \in H^1(\Strip_L)$ is a $C^1(\Strip_L)$ function, $\phi \geq 0$ and $\phi \to 0$ as $|x| \to \infty$. 
     Since $I_{\omega,  0}(\phi) = 0$, it follows that
     \[I_{\omega,\gamma}(\phi) = \gamma \int_0^L |\phi(0,y)|^2 dy > 0.\]
     Let $\tau_z \phi = \phi(x - z, y)$.
     Then $I_{\omega,  0}(\tau_z \phi) = 0$, while   
     \[ I_{\omega,\gamma}(\tau_z \phi) = \gamma\int_0^L |\phi(-z, y)|^2 dy \to 0\]
     as $|z| \to \infty$. 
     Thus, denoting 
     \[
     m_z =  \left(1 + \frac{I_{\omega, \gamma}(\tau_z \psi)}{\| \tau_z \psi\|_{L^{p+1}(\Strip_L)}^{p+1}}\right)^\frac{1}{p-1}, 
     \] 
     we have $m_z \to 1$, $I_{\omega, \gamma}(m_z\tau_z \psi)=0$ and 
     \[ 
     s_{\omega,\gamma} \leq m_z^{\frac{p+1}{p-1}} s_{\omega,0} \to s_{\omega,0} 
     \]
     as $z \to \infty$. 
\end{proof}

We employ the following profile decomposition lemma, see \cite[Theorem 5.1]{JeTa05}. The proof given in \cite{JeTa05} is in the $\mathbb R^d$ setting but is easily generalizable to $\Strip_L$. 
\begin{lem}\label{lemProfilDeco1}
    Let $(u_n)$ be a minimizing sequence for $s_{\omega,\gamma}$. Then, there exist a subsequence, still denoted by $(u_n)$, a solution $u$ to \eqref{eqDeltaStrip}, $k \in \N$, $(x_n^i) \subset \R$ and a set $\{v_i\}_{i = 1,..,k}$ of nontrivial minimizers for $s_{\omega,0}$, such that 
    \begin{align*}
        & u_n \rightharpoonup u_0  \ \mbox{ in } \ H^1(\Strip_L), \\
        & S_\gamma(u_n) \to s_{\omega,\gamma} = S_\gamma(u_0) + k s_{\omega,0},\\
        &u_n - \left(u_0 + \sum_{i=1}^k v_i(x - x_n^i,y)\right) \to 0 \ \mbox{ in } H^1(\Strip_L),\\
        &|x_n^i| \to \infty, \quad |x_n^i - x_n^j| \to \infty \ \mbox{ for } i \neq j.
    \end{align*}
\end{lem}
 We shall also use the following result. 
\begin{lem}\label{lemNotDependY}
    Let $\gamma = 0$. Then there exists $L^\dagger(\omega) >0$ such that for any $0 < L < L^\dagger$, $s_{\omega,0,sym}$ is achieved in a function that does not depend on $y$, which is up to translation and phase shift the profile $\phi_{\omega,0}$ trivially extended in $y$.  
\end{lem}

\begin{proof}
   This result can proved following a similar reasoning to the one of Section \ref{secStripToLine}, we omit the details here. 
\end{proof}

\begin{lem} \label{lemHSmall}
   Let $\gamma >0$,  $\omega > \gamma^2/4$ and   $L^\dagger =L^\dagger(\omega)>0$ given by Lemma \ref{lemNotDependY}. For any $0 < L < L^\dagger$,  $s_{\omega,\gamma,sym}$ defined in \eqref{eqDRadial} is achieved. 
\end{lem}

\begin{proof}
    Let $(u_n)$ be a minimizing sequence for $s_{\omega,\gamma,sym}$. By the definition of $s_{\omega,\gamma,sym}$ and since $I_{\omega,\gamma}(u_n) \leq 0$, we obtain that $u_n$ is uniformly bounded, and converges, up to a subsequence, weakly in $H^1(\Strip_L)$.  Assume by contradiction that there does not exist a strongly convergent subsequence. We apply Lemma \ref{lemProfilDeco1} to this sequence. The case $k = 0$ corresponds to a converging sequence, which we excluded by assumption. By symmetry in $x$, $k$ must be an even number and therefore $k\geq 2$. This implies that 
    \begin{equation*}
        s_{\omega,\gamma,sym} \geq S_\gamma(u_0) + 2s_{\omega,0}
    \end{equation*}
    where we used the notations of Lemma \ref{lemProfilDeco1}. Since $S_\gamma(u_0) \geq 0$, this implies 
    \begin{equation}\label{eqDrUpBnd1}
        s_{\omega,\gamma,sym} \geq 2s_{\omega,0}.
    \end{equation}
    On the other hand,  we have $I_{\omega,\gamma}(\phi_{\omega,\gamma}) = 0$ where $\phi_{\omega,\gamma}$ is defined in \eqref{eqGS1DIntr} and thus 
    \[
    s_{\omega,\gamma,sym} \leq \frac{p-1}{2(p+1)} \| \phi_{\omega,\gamma} \|_{L^{p+1}(\Strip_L)}^{p+1}.
    \] 
    From Lemma \ref{lemNotDependY}, we know that there exists $L^\dagger$ such that $s_{\omega,0,sym} = \frac{p-1}{2(p+1)} \| \phi_{\omega, 0} \|_{L^{p+1}(\Strip_L)}^{p+1}$ for any $0 < L \leq L^\dagger.$ Since  \[ \| \phi_{\omega,\gamma} \|_{L^{p+1}(\Strip_L)}^{p+1} < 2 \| \phi_{\omega, 0} \|_{L^{p+1}(\Strip_L)}^{p+1},\] we obtain a contradiction with \eqref{eqDrUpBnd1}.
\end{proof}

Next, we  show that for any $\omega > 0$, there exists $\gamma^*(\omega,L) >0$ such that for any $ 0 < \gamma < \gamma^*$, there exists a minimizer for \eqref{eqDRadial}.

To give the proof, we fix $\omega >0$, and let $\psi$ be the minimizer for $s_{\omega,0}$. The function $\psi$ is symmetric with respect to $x$, see Appendix \ref{sec:properties}.

We define 
\begin{equation}
    \label{eqPsiTau}
\psi_\sigma(x,y) = \psi(|x| + \sigma,y).
\end{equation}
We start by observing the following.
\begin{lem}\label{lemTau}
    For any $\sigma >0$, we have $I_{\omega,0}(\psi_\sigma) >0$ and $I_{\omega,0}(\psi_{-\sigma}) <0$. 
\end{lem}

\begin{proof}
    We notice that 
    \begin{equation}\label{eq:I0inItau}
        0 = 2I_{\omega,0}(\psi) = I_{\omega,0}(\psi_\sigma)+ I_{\omega,0}(\psi_{-\sigma}).
    \end{equation}
    Indeed, the function $\psi$ is symmetric with respect to $x$. This implies that $|\psi|^p$ and $|\partial_x \psi|^2$ are symmetric with respect to $x$ for any $p \geq 1$. Thus we use the following two identities
    \begin{equation*}
    \begin{aligned}
        \int_{\R} | \psi |^2 dx &= \int_\sigma^\infty |\psi|^2 dx + \int_{-\infty}^\sigma |\psi|^2 dx \\ 
        & = \int_\sigma^\infty |\psi|^2 dx + \int_{-\sigma}^\infty |\psi|^2 dx \\
        & = \int_0^\infty |\psi(x + \sigma) |^2 dx + \int_0^\infty |\psi(x - \sigma) |^2 dx
    \end{aligned}
    \end{equation*}
    and 
    \begin{equation*}
       \int_\R |\psi(|x| + \sigma) |^2 dx = 2\int_0^\infty |\psi(x + \sigma) |^2 dx
    \end{equation*}
    for any $\sigma \in \R$ with their counterparts for the $L^{p+1}(\R)$ and $\dot{H}^1(\R)$-norms to obtain \eqref{eq:I0inItau}.

    Thus it remains to prove that $I_{\omega,0}(\psi_\sigma) >0$ for any $\sigma >0$. If the opposite is true, that is if there exists $\tilde \sigma >0$ such that $I_{\omega,0}(\psi_{\tilde{\sigma}}) \leq 0$, then 
    \[ 
    \| \psi_\sigma \|_{L^{p+1}(\Strip_L)}^{p+1} = 2 \int_0^L \int_{\tilde \sigma}^\infty | \psi|^{p+1} \,dx\,dy < 2 \int_0^L \int_{0}^\infty | \psi|^{p+1} \,dx\,dy  =  \| \psi \|_{L^{p+1}(\Strip_L)}^{p+1}, 
    \]
    which contradicts the definition of $\psi$ as minimizer for $s_{\omega,0}$.
\end{proof}

Next, we prove the following.

\begin{lem}\label{lemGammaStar}
    For any $\omega >0$, there exists $0 < \gamma^*=\gamma^*(\omega,L) < 2 \sqrt{\omega}$, such that for any $0 \leq \gamma < \gamma^*$, there exists $\sigma = \sigma(\gamma) >0$ such that $I_{\omega,\gamma}(\psi_{-\sigma}) \leq 0$. 
\end{lem}

\begin{proof}
We fix $\omega > 0$. Let $\sigma^*(\omega) >0$ be defined by 
    \[ 
    I_{\omega, 0}(\psi_{-\sigma^*}) = \inf_{\sigma> 0} I_{\omega, 0}(\psi_{-\sigma}) < 0,
    \]
where the last inequality follows from Lemma \ref{lemTau}. Now,
    for any $\gamma \in \R$, we have
    \[ I_{\omega,\gamma}(\psi_{-\sigma^*}) = I_{\omega,0}(\psi_{-\sigma^*}) + \gamma \int_0^L |\psi(-\sigma^*,y)|^2 \, dy.\]
     It follows that $I_{\omega,\gamma}(\psi_{-\sigma^*}) \leq 0$ for any $\gamma < \gamma^*$ where
\begin{equation*}
    \gamma^* =  \frac{ - I_{\omega,0}(\psi_{-\sigma^*})}{ \int_0^L |\psi(-\sigma^*,y)|^2 \, dy}.
\end{equation*}
This concludes the proof.
\end{proof}

Next, we prove the existence of a ground state for $\gamma < \gamma^*$.  

\begin{lem}
    \label{lem:gamma-small}
    For $\gamma^*$  defined in Lemma \ref{lemGammaStar}, and for any $0 < \gamma <\gamma^*$, there exists a minimizer for the problem \eqref{eqDRadial}.
\end{lem}

\begin{proof}
    Let $(u_n)$ be a minimizing sequence for $s_{\omega,\gamma,sym}$. Then, as in the proof of Lemma \ref{lemHSmall}, either a subsequence, still denoted by $(u_n)$ converges strongly to some $u$, or the decomposition in Lemma \ref{lemProfilDeco1} is true with $k$ even and $k \geq 2$. We show that the second case yields a contradiction. Indeed, if it is true, then on one hand we obtain $s_{\omega,\gamma,sym} > 2 s_{\omega,0}$. But on the other, let $\psi_\sigma$ be defined as in \eqref{eqPsiTau}. Then, by Lemma \ref{lemGammaStar}, there exists $\sigma^*>0$ such that $\psi_{-\sigma^*}$ verifies $I_{\omega,\gamma}(\psi_{-\sigma^*})\leq 0$. Consequently, by the definition of $s_{\omega,\gamma,sym}$, we have 
    \begin{multline*}
        \frac{2(p+1)}{p-1} s_{\omega,\gamma,sym} \leq \| \psi_{-\sigma^*} \|_{L^{p+1}(\Strip_L)}^{p+1} = 2 \int_0^L \int_{-\sigma^*}^\infty |\psi|^{p+1}dxdy 
        \\
         <   2 \int_0^L \left(\int_{-\sigma^*}^\infty |\psi|^{p+1}dx + \int_{-\infty}^{-\sigma^*} |\psi|^{p+1}dx\right) dy = 2 \| \psi \|_{L^{p+1}(\Strip_L)}^{p+1} = 2 \frac{2(p+1)}{p-1} s_{\omega,0}.
    \end{multline*}
       Thus we reach a contradiction.  
\end{proof}

\begin{proof}[Proof of Theorem \ref{thm:rep_action}]
    The theorem is a direct consequence of Lemma \ref{lemHSmall} and Lemma \ref{lem:gamma-small}.
\end{proof}

\section{Existence of energy minimizers}\label{secEnergyMin}

In this section, we show that solutions to \eqref{eqDeltaStrip} exist as minimizers of the energy under a constraint on the $L^2(\Strip_L)$-norm, i.e as solutions of the minimization problem \eqref{eqMinEnergyMass}.

As before, we divide the rest of the section into two parts depending on the sign of $\gamma$. In the case $\gamma < 0$, the \textit{run-away} soliton is energetically non-favorable compared to the concentrated one and we establish the existence of an energy ground state for any given mass. Conversely, in the case $\gamma>0$, we show the occurrence of run-away behavior. 

\subsection[The attractive case]{Energy minimizers in the attractive case}

\begin{lem}\label{lemINegative}
    Let $m > 0$, $\gamma < 0$ and $1 < p < 3$. Then $-\infty < e_{m,\gamma} < 0$.
\end{lem}

\begin{proof}
    As in Lemma \ref{lemTestComput}, we use the function  $f_\gamma(x,y) = \sqrt{\frac{-\gamma}{2L}} e^{\frac{\gamma|x|}{2}}$. We have  $ \| \sqrt{m} f \|_{L^2(\Strip_L)}^2 = m$ and 
    \[
    E_\gamma(\sqrt{m}f) = - \frac{m \gamma^2}{4} - \frac{m^\frac{p+1}{2} 2^{\frac{3-p}{2}}}{(p+1) L^\frac{p-1}{2}} (-\gamma)^{\frac{p-1}{2}} < 0.
    \]    
   On the other hand, Gagliardo-Nirenberg inequality and Lemma \ref{lemLambdaGamma} imply that $-\infty < e_{m,\gamma}$. 
\end{proof}

\begin{lem}\label{lemImContinuo}
    The function $m \to e_{m,\gamma}$ is continuous.
\end{lem}

\begin{proof}
    Let $m>0$ and $(m_n) \subset (0, \infty)$ be such that $m_n \to m$ as $n \to \infty$. Let $(u_n)$ be a minimizing sequence for $e_{m,\gamma}$. Then it follows 
    \begin{multline}
        e_{m_n,\gamma} \leq E_\gamma \left(\frac{\sqrt{m_n}}{\sqrt{m}} u_n \right) = E(u_n)\\
         + \frac{m_n}{m(p+1)}  \left( 1 - \left( \frac{\sqrt{m_n}}{\sqrt{m}}\right)^{p-1}\right) \| u_n \|_{L^{p+1}(\Strip_L)}^{p+1} + \left(\frac{m_n}{m}- 1\right) E_\gamma(u_n)
        \label{eq:doubled-inequality}
    \end{multline}
    and thus 
    \[ 
    \limsup_{n \to \infty} e_{m_n,\gamma} \leq e_{m,\gamma}.
    \] 
    On the other hand, we fix a sequence $v_n$ such that $M(v_n) = m_n$ and $E(v_n) \leq e_{m_n,\gamma} + \epsilon_n$ with $\epsilon_n \to 0$. Then, as in  \eqref{eq:doubled-inequality}, we obtain
    \begin{equation*}
        e_{m,\gamma}
        \leq e_{m_n,\gamma} + \epsilon_n + \frac{m}{m_n(p+1)}  \left( 1 - \left( \frac{\sqrt{m}}{\sqrt{m_n}}\right)^{p-1}\right) \| v_n \|_{L^{p+1}(\Strip_L)}^{p+1} + \left(\frac{m}{m_n} - 1\right) E(v_n),
    \end{equation*}
    which implies 
    \[
    e_{m,\gamma} \leq \liminf_{n\to \infty} e_{m_n,\gamma}.
    \]
    This concludes the proof.
\end{proof}

The proof of Theorem \ref{ThmEnMinIntro} relies on concentration-compactness arguments. Such arguments are classical in the case of homogeneous spaces but should be adapted when the problem is inhomogeneous as for the fractured strip. 
While classical concentration compactness features only three cases: compactness up to translation, dichotomy, and vanishing, in the present context of a strip with an inhomogeneity we need a version of the concentration compactness principle where the usual compactness case is divided into two sub-possibilities: compactness and run-away sequences. Such a version is provided in the context of graphs in \cite[Lemma 3.3]{AdCaFiNo14} and can be adapted mutatis mutandis to the strip. The statement is the following. 

\begin{lem}\label{lemConcCompc}
Let $m > 0$ and $(u_n)\subset H^1(\Strip_L)$ be such that:
\[
\| u_n \|_{L^2(\Strip_L)}^2 = m, \quad \sup_{n} \|\nabla u_n\|_{L^2(\Strip_L)} < \infty.
\]
Let 
 \begin{equation}\label{eqConcComMu}
     \mu = \lim_{R \to \infty} \liminf_{n \to \infty} \sup_{z \in \Strip_L} \int_{B(z,R)} |u_n|^2 dx \in [0,m]
 \end{equation}
 where $B(z,R) = \{z' \in \Strip_L\, : \, |z' - z| < R \}$.
Then there exists a subsequence $(u_{n_k})$ such that:

\begin{itemize}
    \item[(i) (Compactness)] If $\mu = m$ then one of the following occurs:
    \begin{itemize}
        \item[(i1) (Convergence)] There exists a function $u \in E$ such that $u_{n_k} \to u$ in $L^p$ as $k \to \infty$ for all $2 \leq p < \infty$.
        \item[(i2) (Runaway)] $ \forall R > 0$ and  $2 \leq p$,  
        \[
        \|u_{n_k}\|_{L^p(B(0,R))} \to 0. 
        \]
    \end{itemize}
    \item[(ii) (Vanishing)] If $\mu = 0$, then $u_{n_k} \to 0$ in $L^p$ as $k \to \infty$ for all $2 < p < \infty$.
    \item[(iii) (Dichotomy)] If $0 < \mu < m$, then there exist two sequences $\{V_k\}_{k \in \mathbb{N}}$ and $\{W_k\}_{k \in \mathbb{N}}$ such that
    \[
    \operatorname{supp} V_k \cap \operatorname{supp} W_k = \emptyset, 
    \]
    \[
    \|V_k\|_{H^1(\Strip_L)} + \|W_k\|_{H^1(\Strip_L)} \leq c \|u_{n_k}\|_{H^1(\Strip_L)}, 
    \]
    \[
    \lim_{k \to \infty} \| V_k\|_{L^2(\Strip_L)}^2 = \mu, \quad \lim_{k \to \infty} \| W_k\|_{L^2(\Strip_L)}^2 = m - \mu, 
    \]
    \[
    \liminf_{k \to \infty} \left(\|\nabla \Psi_{n_k}\|^2_{L^2(\Strip_L)} - \|\nabla V_k\|^2_{L^2(\Strip_L)} - \|\nabla W_k\|^2_{L^2(\Strip_L)}\right) \geq 0, 
    \]
    \[
    \lim_{k \to \infty} \left(\|u_{n_k}\|_{L^p(\Strip_L)}^p - \|V_k\|_{L^p(\Strip_L)}^p - \|W_k\|_{L^p(\Strip_L)}^p\right) = 0, \quad 2 \leq p < \infty.
    \]
\end{itemize}

\end{lem}

\begin{proof}[Proof of Theorem \ref{ThmEnMinIntro}] 
Let $(u_n) \subset H^1(\Strip_L)$ be a minimizing sequence for $e_{m,\gamma}$. Since $e_{m,\gamma} < 0$, we have $E_\gamma(u_n) \leq 0$ for $n$ large enough. Thus, using $E(u_n) \leq 0$,  Gagliardo-Nirenberg inequality, the trace estimate \eqref{eqH12Control} and Young inequality, we have 
\[
\begin{aligned}
    \| \nabla u_n \|_{L^2(\Strip_L)}^2 &\lesssim \|  u_n \|_{L^{p+1}(\Strip_L)}^{p+1} + \int_0^1 |u_n(0,y)|^2 dy  \\ 
    & \lesssim \|  u_n \|_{L^{2}(\Strip_L)}^{2}  \|  \nabla u_n \|_{L^{2}(\Strip_L)}^{p-1}+ \int_0^L \int |u_n| |\nabla u_n| dx dy \\
    & \lesssim  \|  \nabla u_n \|_{L^{2}(\Strip_L)}^{p-1} + \eps \| \nabla u_n \|_{L^{2}(\Strip_L)}^{2} + \frac{1}{\eps} \| u_n \|_{L^{2}(\Strip_L)}^{2} 
\end{aligned}
\]
for any $\eps>0$.  Since $p-1 < 2$, choosing $\epsilon$ small enough we can prove that $(\| u_n \|_{H^1(\Strip_L)})$ is bounded. Thus, up to a subsequence, there exists $u \in H^1(\Strip_L)$ such that $u_n \rightharpoonup u$ in $H^1(\Strip_L)$.

 We employ the concentration-compactness theorem to show that $u_n \to u \neq 0$ strongly in $H^1(\Strip_L)$. Let $\mu$ be defined as in \eqref{eqConcComMu}. We show that $\mu = m$.

\textit{Step $1$:} We have 
\begin{equation}
    \label{eqLiminf1}
    \liminf_{n \to \infty} \| u_n \|_{L^{p+1}(\Strip_L)} >0.
\end{equation} 
Indeed, assume the opposite. Then, by weak lower semicontinuity, we obtain  $\| u \|_{L^{p+1}(\Strip_L)} \leq \liminf_{n \to \infty} \| u_n \|_{L^{p+1}(\Strip_L)} = 0$ and consequently $u_n \rightharpoonup 0$ in $H^1(\Strip_L)$. By Lemma \ref{lemTraceComp}, we obtain 
\[
\gamma \int_0^L |u_n(0,1)|^2 dy \to 0.
\]
By Lemma \ref{lemINegative} we have $e_{m,\gamma} < 0$ which implies 
\[ 
0 > \lim_{n\to \infty} E_\gamma(u_n) > - \frac{1}{p+1} \lim_{n\to \infty} \| u_n \|_{L^{p+1}(\Strip_L)}^{p+1},
\]
which is a contradiction. 

\textit{Step $2$:} By contradiction, assume that $\mu = 0$. Then there exists a subsequence $(u_{n_k})$ such that $u_{n_k} \to 0$ in $L^{p+1}(\Strip_L)$ (see e.g. \cite[Proposition 1.7.6]{Ca03}).
This contradicts \eqref{eqLiminf1} and we get $\mu >0$. 

\textit{Step $3$:} For any $\mu \in (0,m)$, we have \begin{equation}
    \label{eqIm/m}
    \frac{e_{\mu,\gamma}}{\mu} > \frac{e_{m,\gamma}}{m}. 
\end{equation}  Let $(v_n)$ be a minimizing sequence for $e_{\mu,\gamma}$. Then we have 
 \begin{equation*}
        e_{m,\gamma}\leq E_\gamma\left(\frac{\sqrt{m}}{\sqrt{\mu}} v_n\right) =\frac{m}{\mu} E_\gamma(v_n) + \frac{m}{\mu(p+1)}  \left( 1 - \left( \frac{\sqrt{m}}{\sqrt{\mu}}\right)^{p-1}\right) \| v_n \|_{L^{p+1}(\Strip_L)}^{p+1}.
\end{equation*}
Taking the limit, we obtain
        \begin{equation*}
          e_{m,\gamma}          \leq \frac{m}{\mu} e_{\mu,\gamma} + \frac{m}{\mu(p+1)}  \left( 1 - \left( \frac{\sqrt{m}}{\sqrt{\mu}}\right)^{p-1}\right) \liminf_n \| v_n \|_{L^{p+1}(\Strip_L)}^{p+1}. 
    \end{equation*}
Exploiting the equivalent of \eqref{eqLiminf1} for $(v_n)$, we obtain \eqref{eqIm/m}. 
    
    \textit{Step $4$:}
     Suppose $0 < \mu < m$ and let $\theta = \| u \|_{L^2}^2 \in [0,\mu]$. 
     We notice that $\theta = 0$ is ruled out in the next step as it will yield $u_{n} \rightharpoonup 0$ in $L^2(\Strip_L)$. Then by the Brezis-Lieb inequality and Lemma \ref{lemTraceComp} we get 
    \begin{align} 
        \| u_n - u \|_{L^2(\Strip_L)}^2 = \| u_n \|_{L^2(\Strip_L)}^2 - \| u \|_{L^2(\Strip_L)}^2 + o(1) = m - \theta + o(1), \nonumber \\
        \label{eqEq11} E_\gamma(u_n) = E_\gamma(u_n - u) + E_\gamma(u) + o(1) \geq e_{m - \theta+ o(1), \gamma} + e_{\theta,\gamma} + o(1).
    \end{align}

    By Lemma \ref{lemImContinuo} and \eqref{eqEq11}, we get 
    \[ e_{m,\gamma} \geq e_{m - \theta,\gamma} + e_{\theta,\gamma}.\]
    After a scaling, and using property \eqref{eqIm/m}, we obtain 
    \[ e_{m,\gamma} > \frac{m - \theta}{m} e_{m,\gamma}+ \frac{\theta}{m} e_{m,\gamma} = e_{m,\gamma}\]
    which is absurd.

    \textit{Step $5$:} It remains to rule out the \textit{run-away} behavior, that is the case $\mu = m$ and such that there exists a subsequence $(u_{n_k})$ such that 
    \begin{equation*}
        \int_{B(0,R)} | u_{n_k} |^p dx \to 0 
    \end{equation*}
    for any $p \in [2,\infty)$ and any $R > 0$. By contradiction, suppose that this is the case. Then $u_{n_k} \rightharpoonup  0$ in $L^2(\Strip_L)$. Indeed if $u_{n_k} \rightharpoonup u \not \equiv 0$, then 
    \begin{equation*}
      0<\int_{B(0,R)} | u |^2 dx \leq \lim_{n\to \infty} \left| Re  \int_{B(0,R)}  u_{n_k} \bar u  dx \right| \leq \lim_{n\to \infty} \left( \int_{B(0,R)} | u_{n_k} |^2 \right)^{1/2} \| u \|_{L^2(\Strip_L)} \to 0
      \end{equation*}
     for any $R >0$, which is absurd. 

     We show that $u_{n_k} \rightharpoonup 0$ in $H^1(\Strip_L)$ leads to a contradiction. Indeed, since 
     \begin{equation*}
         E_{\gamma} (u_{n_k}) = E_{0} (u_{n_k}) + \frac{\gamma}{2} \int_0^L |u_{n_k}(0,y)|^2 dy
     \end{equation*}
      then, by Lemma \ref{lemTraceComp}, we obtain that 
     \begin{equation*}
         e_{m,\gamma} = \lim_{k \to \infty} E_{\gamma}(u_{n_k}) =  \lim_{k \to \infty} E_{0} (u_{n_k}) \geq E_{0} (\psi_m)
     \end{equation*}
     where $\psi_m$ is the ground state for the minimizing problem $e_{m,0}$, centered in $x = 0$. The existence of $\psi_m$ can be shown following the same arguments as in \cite[Appendix]{TeTzVi14}. We reach a contradiction since 
     \begin{equation*}
         e_{m,\gamma} \leq E_{\gamma} (\psi_m) < E_{0} (\psi_m).
     \end{equation*}
     Thus, we obtain $\mu = m$, $u_n \to u$ in $L^2(\Strip_L)$ and, by interpolation,  $u_n \to u$ in $L^{p+1}(\Strip_L)$. Moreover 
     \[ 
     \gamma\int_0^L|u_n(0,y)|^2 dy \to \gamma\int_0^L|u(0,y)|^2 \, dy 
     \] 
     by Lemma \ref{lemTraceComp}. Since 
     \[\| \nabla u \|_{L^2(\Strip_L)}^2 \leq \liminf_{n \to \infty} \| \nabla u_n \|_{L^2(\Strip_L)}^2,\] 
     the definition of $e_{m,\gamma}$ implies that $u_n \to u$ in $H^1(\Strip_L)$. This proves the existence of a minimizer for \eqref{eqMinEnergyMass}. Replacing $u$ by $|u|$, we still have a minimizer, hence we can assume without loss of generality that $u\geq 0$. The existence of $\omega(m)$ follows from a classical Lagrange multiplier argument. 
\end{proof}

\subsection[The repulsive case]{Energy minimizers in the repulsive case}\label{secEnMinRepul}

In the case $\gamma >0$, the energy minimizers with fixed mass in the whole space $H^1(\Strip_L)$ do not exist, as shown in the following lemma.

\begin{lem}
    For any $m>0$ and $\gamma >0$,  problem \eqref{eqMinEnergyMass} does not admit a minimizer.
\end{lem}

\begin{proof}
    Suppose, by contradiction, that a minimizer  $\psi_{m,\gamma} \geq 0$ for $e_{m,\gamma}$ with $\gamma >0$ exists. As a minimizer, this profile satisfies the elliptic equation \eqref{eqDeltaStrip} and thus enjoys the properties shown in Appendix \ref{sec:properties}. Moreover, $\psi_{m,\gamma}(0,y) = 0$ for almost every $y$ as 
\begin{equation*}
    \int_0^L |\psi_{m,\gamma}(0,y)|^2 = 0.
\end{equation*}
Indeed, by contradiction, suppose that this is not true. Then there exists $T>0$ such that for any $\tau > T$
\begin{equation*}
    E_{\gamma} (\psi_{m,\gamma}(. - \tau,.)) < E_{\gamma} (\psi_{m,\gamma})
\end{equation*}
by the exponential decay in $x$. This contradicts the definition of  $\psi_{m,\gamma}$. 

We can create a family of minimizers in the following way: for any $t >0$, we define 
\begin{equation*}
    \psi_{m,\gamma,t}(x,y) = \begin{cases} 0 \quad x \in [-t,t]\\ 
        \psi_{m,\gamma}(x - t,y) \quad  x \geq t \\
         \psi_{m,\gamma}(x + t,y) \quad  x \leq -t.
    \end{cases}
\end{equation*}
Then for any $t >0$, $\psi_{m,\gamma,t}(x,y) \in C(\Strip_L) \cap H^1((\Strip_L)$, $\psi_{m,\gamma,t}(x,y) \geq 0$ and it is a minimizer for $e_{m,\gamma}$. In particular, $\psi_{m,\gamma,t}(x,y)$ also satisfies \eqref{eqDeltaStrip}. This implies that $\psi_{m,\gamma,t}(x,y) > 0$ for any $x \neq 0, y \not \in \{ 0,L \}$ by Harnack's inequality, which is a contradiction.  
\end{proof}

It would therefore be natural to search for a minimizer of the problem restricted to symmetric functions $e_{m,\gamma}^{sym}$ defined in \eqref{eqEnMasRadIntr}. As we have seen for action minimizers in Section \ref{secActionMinRepuls}, the lack of explicit expressions for the candidate minimizers renders the analysis delicate, and restrictions on the parameters $\gamma$ or $L$
 had to be made. Similar restrictions are needed for energy minimization. We do not provide a fully rigorous result but just give a few comments on what is expected.

Using equation \eqref{eqIm/m} we give some idea on the restriction on $\gamma$ in terms of $m$. In the symmetric case, the run-away on one side cannot happen, and on both sides, the best option for energy minimization is to escape as two ground states for the problem $e_{0,m/2}$. We denote by $\psi_{m/2}$ a minimizer of $e_{0,m/2}$. Thus one ground state should exist if there exists a symmetric test function $\phi_m \in H^1(\Strip_L)$ with mass $m$ and 
\begin{equation*}
    E_\gamma (\phi_m) < 2E_0(\psi_{m/2}).
\end{equation*}
Such a function exists for any $m>0$ when $\gamma>0$ is small enough. Indeed  \eqref{eqIm/m} yields 
\begin{equation*}
    E_0 (\psi_m) < 2E_0(\psi_{m/2}).
\end{equation*}
For $\gamma>0$ small enough, we have
\begin{equation*}
    \gamma \int_0^L |\psi_m(0,y)|^2 dy < 2E_0(\psi_{m/2}) - E_0 (\psi_m),
\end{equation*}
and therefore $\psi_m$ verifies
\[
 E_\gamma (\psi_m) < 2E_0(\psi_{m/2}),
\]
which gives us the desired test function.

\section{The shrinkage limit}\label{secStripToLine}

In this section, we show that the normalized energy ground states in the attractive case do not depend on the transverse variable $y \in [0, L]$ when the amplitude $L$ is small enough.

We start by rescaling the problem in such a way that we work on a fixed-length strip. 
For any fixed length $L > 0$ and $\psi \in H^1(\Strip_L)$, define $u \in H^1(\Strip)$ by $\psi(x,y) = u(x, y/L)$. Then we have
\begin{equation*}
\begin{gathered}
         E_\gamma(\psi) = L \int_0^1 \int_\R \left( \frac{1}{2}  |\partial_x u|^2 + \frac{1}{2L^2}|\partial_y u|^2 - \frac{1}{p+1}|u|^{p+1} \right) dxdy + \frac{\gamma L}{2}\int_0^1 |u(0,y)|^2 dy, \\
     \| \psi \|_{L^2(\Strip_L)}^2 = L \| u \|_{L^2(\Strip)}^2.
\end{gathered}
\end{equation*}
For $u\in H^1(\Strip)$, we introduce the notation 
\[
\tilde{E}_{L, \gamma} (u) =  \int_0^1 \int_\R \left( \frac{1}{2}  |\partial_x u|^2 + \frac{1}{2L^2}|\partial_y u|^2 - \frac{1}{p+1}|u|^{p+1} \right) dxdy + \frac{\gamma }{2}\int_0^1 |u(0,y)|^2 dy.
\]
For any $m >0$ the minimizing problem $e_{m,\gamma}$ defined in \eqref{eqMinEnergyMass} on the strip $\Strip_L$ of length $L$ is equivalent on the normalized strip $\Strip$ to
\begin{equation*}
    \inf\left \{ \tilde{E}_{L, \gamma} (u) :  u \in H^1(\Strip), \, \| u \|_{L^2(\Strip)}^2 = \frac{m}{L} \right\}.
\end{equation*}
The existence of non-negative minimizers to this second problem can be deduced from the same arguments as in Theorem \ref{ThmEnMinIntro}. For the sake of simplicity, we normalize $\frac mL$ to $1$ by choosing $m = L$ and   we consider the normalized problem 
\begin{equation}\label{eqEtilde}
    \tilde e_{1,\gamma}(L) = \inf \{ \tilde {E}_{L,\gamma} (u): u \in H^1(\Strip), \, \| u \|_{L^2(\Strip)}^2 = 1 \}.
\end{equation}

Our scope is to first show that the non-negative minimizers of the problem \eqref{eqEtilde} are given by functions that do not depend on the transverse variable $y \in [0,1]$ for $L>0$  small enough. We will also identify precisely the minimizers as being the one-dimensional ground state extended to the strip $\Strip$.

First, we show that, as the length approaches zero, the energy levels of the minimizers of $\tilde e_{1,\gamma}(L)$ approach those of the one-dimensional normalized ground state.  To this end, recall that we denote by $e^{1D}_{1,\gamma}$ the energy of the normalized one-dimensional energy ground state (see \eqref{eqEnMin1D}).
As stated in Section \ref{sec1dGs}, there exists a unique real-valued and positive profile minimizing $e^{1D}_{1,\gamma}$. It is given by the %
profile  $\phi_{\tilde \omega, \gamma}$ defined in formula \eqref{eqGS1DIntr} with $\tilde \omega = \tilde \omega(\gamma)$ that can be explicitly computed using formulas \eqref{eqMtoOmega1D} and \eqref{eqCOmega}.   We extend this profile to the strip in the following way. We define $\tilde {\phi}_{ \tilde \omega,\gamma}(x,y) = \phi_{\tilde \omega, \gamma}(x)$ for $(x,y) \in \Strip$. By construction, we have 
    \begin{equation}\label{eqEnMass1dGS}
        \tilde E_L(\tilde {\phi}_{ \tilde \omega,\gamma}) = E^{1D}_\gamma (\phi_{\tilde \omega, \gamma}), \quad \| \tilde {\phi}_{ \tilde \omega,\gamma} \|_{L^2(\Strip)}^2 = 1.
    \end{equation}
Moreover, by Theorem \ref{ThmEnMinIntro}, for any $L >0$, there exists a real-valued and non-negative minimizer of the problem \eqref{eqEtilde} which we denote by $u_{L}$, i.e.
\[
u_L\in H^1(\Strip),\quad \tilde E_{L,\gamma}(u_{L})=\tilde e_{1,\gamma}(L),\quad \norm{u_L}_{L^2(\Strip)}=1.
\]
In the next lemma, we show that $\tilde e_{1,\gamma}(L) \to e^{1D}_{1,\gamma}$ as $L \to 0$.

\begin{lem} For any $\gamma <0$ and $1 < p < 3$, we have
    \begin{equation}
        \label{eq1D2DMinEn}
       \lim_{L\to 0} \tilde e_{1,\gamma}(L) = e^{1D}_{1,\gamma}.
    \end{equation}
    Moreover, the positive minimizers $u_L$ of $\tilde e_{1,\gamma}(L)$ verify 
    \begin{equation}
        \label{eqKinEnYH0}
        \lim_{L\to 0} \frac{1}{L^2} \| \partial_y u_L \|_{L^2(\Strip)}^2 = 0.
    \end{equation}
\end{lem}

\begin{proof}
    On one hand, from \eqref{eqEnMass1dGS}, we obtain the upper bound 
    \begin{equation}
        \label{eq1d2dIneq}
        \tilde e_{1,\gamma}(L) \leq  \tilde E_{L,\gamma}(\tilde {\phi}_{ \tilde \omega,\gamma}) = e^{1D}_{1,\gamma}.
    \end{equation} 
    Since $u_L$ minimizes $\tilde e_{1,\gamma}(L)$, by \eqref{eq1d2dIneq}, the Gagliardo-Nirenberg inequality and the Trace inequality  \eqref{eqH12Control} we obtain 
    \begin{equation*}
        \begin{aligned}
            \| \partial_x u_L \|_{L^2(\Strip)}^2  + \frac{1}{L^2} \| \partial_y u_L \|_{L^2(\Strip)}^2 & \leq 2 e^{1D}_{1,\gamma} +  \frac{2}{p+1} \| u_L \|_{L^{p+1}(\Strip)}^{p+1} - \gamma \int_0^1 |u_L(0,y)|^2 dy \\
            & \leq 2 e^{1D}_{1,\gamma} + \frac{2C}{p+1} \| u_L \|_{L^2(\Strip)}^2  \| \nabla u_L \|_{L^2(\Strip)}^{p-1} + \gamma C \| u_L \|_{L^2(\Strip)}\| \nabla u_L \|_{L^2(\Strip)}.
        \end{aligned}
    \end{equation*}
Since $\| u_L \|_{L^2(\Strip)} = 1$, by Young inequality, we arrive to
\begin{equation*}
        \begin{aligned}
            \| \partial_x u_L \|_{L^2(\Strip)}^2  + \frac{1}{L^2} \| \partial_y u_L \|_{L^2(\Strip)}^2 & \leq C + \frac{1}{\eps} + C\eps \| \nabla u_L \|_{L^2(\Strip)}^2
        \end{aligned}
    \end{equation*}
    for some $\eps >0$. Choosing $\eps$ so that $C\eps \leq \min \left(1, \frac{1}{L^2}\right)$ we obtain that there exists $C > 0$ such that
    \begin{equation}\label{eqKinEnY0}
        \sup_{L \in [0, \infty)} \frac{1}{L^2} \| \partial_y u_L \|_{L^2(\Strip)}^2 \leq C.  
    \end{equation}
    On the other hand, we observe that for any $L >0$ we have
    \begin{align}
        \label{eqLWRBND1}
        &\tilde e_{1,\gamma}(L) \geq \int_0^1 E^{1D}_{\gamma}(u_L(y)) \, dy, \quad \int_0^1 M^{1D}(u_L(y)) dy = 1,
    \end{align}
    where $E^{1D}_{\gamma}$ and $M^{1D}$ are the one-dimensional energy and mass, defined respectively in \eqref{eqEn1D} and \eqref{eqMass1D}. We define $m_L(y) =  M^{1D}(u_L(y)) \in L^1(0,1)$. We denote by $\phi_{\omega_L(y), \gamma}$ the one-dimensional soliton defined in \eqref{eqGS1DIntr} with $\omega_L(y)$ adjusted so that it  satisfies 
    \begin{equation*}
        \int_\R |\phi_{\omega_L(y), \gamma}(x)|^2 dx = m_L(y).
    \end{equation*}
    Notice that $\omega_L(y)$ depends only on $m_L(y)$ and $\gamma$ by formula \eqref{eqMtoOmega1D}.
    Then, by Proposition \ref{prp1DIntr} we have 
    \begin{equation}
      \int_0^1 E^{1D}_{\gamma}(u_L(y)) \, dy \geq \int_0^1 E_{\gamma}^{1D} (\phi_{\omega_L(y), \gamma}) dy.
    \end{equation}
   Exploiting equation \eqref{eqEnPhiOmega} we obtain 
   \begin{equation}\label{eqLWB2}
       \tilde e_{1,\gamma}(L) \geq \int_0^1 m_L^{\frac{p+3}{5-p}}(y) k_1(m_L(y))  +  m_L^{\frac{4}{5-p}}(y) k_2(m_L(y)) \, dy,
   \end{equation}
    where $k_1, k_2 \in L^\infty([0,1])$ are given by
    \begin{align}
        \label{eqK1} & k_1(m_L(y))  = - \frac{5-p}{ (p+3)Q(\omega_L(y),\gamma)}, \\
        \label{eqK2} & k_2(m_L(y)) = \frac{\gamma (p+1)^3}{4(p+3)Q(\omega_L(y),\gamma)^\frac{2}{p-1}}  \sech^\frac{2}{p-1} \left( \tanh^{-1} \left( \frac{\gamma}{2\sqrt{\omega_L(y)}}\right)  \right) 
    \end{align}
    and $Q(\omega_L(y),\gamma)$ is defined in \eqref{eqCOmega}. Since $\| u_L \|_{L^2(\Strip)} = 1$, we have 
    \begin{equation}
        \label{eqMassAve1}
        \int_0^1 m_L(y) \, dy = 1.
    \end{equation}
   On the other hand, we also have 
    \begin{equation}\label{eqRellich1}
        \int_0^1 \left| \frac{d}{dy} m_L(y) \right| \, dy \leq  \frac{1}{2} \int_0^1 \int_\R | u_L| | \partial_y u_L | \, dx\,dy \lesssim \| u_L \|_{L^2(\Strip)} \| \partial_y u_L\|_{L^2(\Strip)} \to 0
    \end{equation}
    as $L \to 0$ from \eqref{eqKinEnY0}. Thus, by the Rellich compactness theorem, 
    we have
    \begin{equation}
        \label{eqRellich}
        \lim_{L\to 0} \int_0^1 |m_L(y) - 1|^{q} dy \to 0,
    \end{equation}
    that is $m_L \to 1$ in $L^q(0,1)$ for any $q \in [1,\infty)$ as $L\to \infty$.
    Now, from \eqref{eqLWB2}, \eqref{eqK1}, \eqref{eqK2} and \eqref{eqRellich} we obtain that 
    \begin{equation*}
    \begin{aligned}
        \lim_{L\to 0} \tilde e_{1,\gamma}(L) &\geq \lim_{L\to 0 }  \int_0^1 m_L^{\frac{p+3}{5-p}}(y) k_1(m_L(y))  +  m_L^{\frac{4}{5-p}}(y) k_2(m_L(y)) \, dy \\ 
        & =  \int_0^1 E^{1D}_\gamma(\tilde \phi_{\tilde \omega,\gamma}(y)) \, dy = e^{1D}_{1,\gamma}.
    \end{aligned}
    \end{equation*}
    Combined with \eqref{eq1d2dIneq}, this proves \eqref{eq1D2DMinEn}.
    Finally, we prove \eqref{eqKinEnYH0}. We observe that 
     \begin{equation*}
        \tilde e_{1,\gamma}(L) = \int_0^1 E^{1D}_\gamma(u_L) \, dy+ \frac{1}{2L^2} \| \partial_y u_L \|_{L^2(\Strip)}^2 \geq \int_0^1 E_{\gamma}^{1D} (\phi_{m_L(y),\gamma}) dy + \frac{1}{2L^2} \| \partial_y u_L \|_{L^2(\Strip)}^2 .
    \end{equation*}
    Thus 
    \begin{equation*}
        \lim_{L \to 0 } \frac{1}{L^2}\| \partial_y u_L \|_{L^2(\Strip)}^2 \leq 2 \lim_{L \to 0 } \tilde e_{1,\gamma}(L) - 2 \lim_{L \to 0 } \int_0^1 E^{\gamma}_{1D} (\phi_{m_L(y),\gamma}) dy  = 0.
    \end{equation*}
    This finishes the proof.
\end{proof}

From now on, we take a sequence $(L_n)$, $L_n \to 0^+$ and a corresponding sequence of positive minimizers $(u_n)=(u_{L_n})$. 

In the next lemma, we show that up to a subsequence, the minimizers of the normalized problem on the strip $\Strip$ converge strongly in $H^1(\Strip)$ to the transversely extended soliton $\tilde \phi_{\tilde \omega,\gamma}$.

\begin{lem}
    Let $\gamma < 0$. There exist a subsequence of $(u_n)$, still denoted by $(u_n)$, such that  
    \begin{equation}\label{eqConverg1}
        \lim_{n\to \infty} \| u_n - \tilde \phi_{\tilde \omega,\gamma} \|_{H^1(\Strip)} = 0.
    \end{equation}
\end{lem}

\begin{proof}
For any $n$, $u_n = u_{L_n}$  is a positive minimizer of \eqref{eqEtilde} and as such it satisfies the elliptic equation 
\begin{equation}\label{eqUnEq1}
        -\partial_{xx} u_n - \frac{1}{L_n^2} \partial_{yy} u_n + \omega_n u_n + \gamma \deltastripunit u_n -  u_n^{p} = 0,
    \end{equation}
for some Lagrange multiplier $\omega_n > \frac{\gamma^2}{4}$. From the definition and \eqref{eq1d2dIneq} it follows that 
$(u_n)$ is bounded in $H^1(\Strip)$. Therefore, there exists a subsequence, still denoted by $(u_n)$ and there exists $u \in H^1(\Strip)$ such that $u_n \rightharpoonup u$ in $H^1(\Strip)$. We observe that  $u \geq 0$ and that $u$ is independent of the transverse variable $y$ from \eqref{eqKinEnYH0}. From \eqref{eq1D2DMinEn} we get $E_{L_n}(u_n) \to e^{1D}_{1,\gamma}< 0$, it follows that 
    \begin{equation}
        \label{eqLwbPotEn1}
        \liminf_{n\to\infty} \| u_n \|_{L^{p+1}(\Strip)}^{p+1} > 0.
    \end{equation}
    As in the proof of Theorem \ref{ThmEnMinIntro}, this implies that the sequence $(u_n)$ does not vanish and $\| u \|_{L^2(\Strip)} >0$. 
    
    By the Pohozaev identities, we obtain
    \begin{equation}
        \label{eq2DPoho}
        \begin{aligned}
             &  \|\partial_x u_n\|^2_{L^2(\Strip)}  + \frac{1}{L_n^2} \|\partial_y u_n\|^2_{L^2(\Strip)} +  \omega_n  - \|u\|^{p+1}_{L^{p+1}(\Strip)} +  \gamma \int_0^1|u_n(0,y)|^2 dy   = 0,\\
         &   \|\partial_x u_n\|^2_{L^2(\Strip)}  - \frac{1}{L_n^2} \|\partial_y u_n\|^2_{L^2(\Strip)} -  \omega_n  + \frac{2}{p+1} \|u\|^{p+1}_{L^{p+1}(\Strip)}  = 0,
        \end{aligned}
    \end{equation}
    which implies after a rearrangement that
    \begin{equation}
        \label{eqOmegaN1}
        \omega_n = - \frac{2(p+3)}{5-p} \tilde {E}_{L,\gamma} (u_n) + \frac{p-1}{5-p} \gamma \int_0^1 |u_n(0,y)|^2 dy - \frac{4}{5-p} \frac{1}{L_n^2} \| \partial_y u_n \|_{L^2(\Strip)}^2.
    \end{equation}
    In particular, by Lemma \ref{lemTraceComp}, \eqref{eq1D2DMinEn} and \eqref{eqKinEnYH0} we obtain 
    \begin{equation} \label{eqOmegaNLim1}
        \omega_n \to - \frac{2(p+3)}{5-p} e^{1D}_{1,\gamma}+ \frac{p-1}{5-p} \gamma |u(0)|^2  =: \omega_\infty.
    \end{equation}
    By combining  \eqref{eqOmegaNLim1}, \eqref{eqKinEnYH0} and passing to the limit in the distributional sense in \eqref{eqUnEq1}, we see that $u \in H^1(\Strip)$ satisfies for any $y \in [0,1]$ 
    \begin{equation}
        \label{eqUEq1}
        -\partial_{xx} u + \omega_\infty u + \gamma \delta_0 u - |u|^{p-1} u = 0, \quad \quad u(x) \geq 0, \quad u \not\equiv 0.
    \end{equation}
    By the uniqueness of positive solutions to \eqref{eqUEq1} we can conclude that $u(x,y) = \phi_{\omega_\infty,\gamma}(x)$, where $\phi_{\omega_\infty,\gamma}$ is the one-dimensional ground state defined in \eqref{eqGS1DIntr}. From \eqref{eqPoho1D} we get
    \begin{equation}
        \label{eqPohozaevId1}
        (5 - p) \omega_\infty \| \phi_{\omega_\infty,\gamma} \|_{L^2(\Strip)}^2 = - 2 (p+3)E_\gamma^{1D}(\phi_{\omega_\infty,\gamma})  - (p-1) \gamma |\phi_{\omega_\infty,\gamma}(0)|^2 .
    \end{equation}
   We recall that $e^{1D}_{1,\gamma} = E^{1D}(\phi_{\tilde {\omega},\gamma})$. Combining \eqref{eqOmegaNLim1} and \eqref{eqPohozaevId1}, we obtain 
   \begin{equation*}
       (5 - p)\omega_\infty (1 - \| \phi_{\omega_\infty,\gamma} \|_{L^2(\Strip)}^2) + 2(p+3) (E_\gamma^{1D}(\phi_{\tilde {\omega},\gamma}) - E_\gamma^{1D}(\phi_{\omega_\infty,\gamma})) = 0.
   \end{equation*}
   Exploiting \eqref{eqEnMassZero}, we get
   \begin{equation*}
       (5 - p)( \omega_\infty - \tilde\omega ) + (p-1)\gamma ( | \phi_{\tilde{\omega},\gamma}(0)|^2 - | \phi_{\omega_\infty,\gamma}(0)|^2) = 0.
   \end{equation*}
   By \eqref{eqGS1Din0},  it follows that
   \begin{equation}
       \label{eqOmegaDiff}
        (5 - p)( \omega_\infty - \tilde\omega ) + (p-1)\gamma \left( \frac{p+1}{2} \right)^{\frac{2}{p-1}} \left(\left( \tilde \omega - \frac{\gamma^2}{4} \right)^{\frac{2}{p-1}} -  \left( \omega_\infty - \frac{\gamma^2}{4} \right)^{\frac{2}{p-1}}\right) = 0.
   \end{equation}
   Since $\gamma < 0$, \eqref{eqOmegaDiff} implies that $\omega_\infty = \tilde \omega$.

   Thus we have obtained that $u_n \rightharpoonup \phi_{\tilde {\omega},\gamma}$ in $H^1(\Strip)$, $u_n \to \phi_{\tilde {\omega},\gamma}$ in $L^2(\Strip)$ and $\omega_n \to \tilde \omega$.  Finally, from \eqref{eq2DPoho} and \eqref{eqPoho1D} we have
   \begin{equation*}
       \begin{aligned}
           \frac{p-1}{p+1} \| \partial_x u_n \|_{L^2(\Strip)}^2 &= 2 E^{2D}_\gamma(u_n) + \frac{2}{p+1} \| u_n \|_{L^2(\Strip)}^2 - \frac{p-1}{p+1} \left( \frac{1}{L_n^2} \| u_n \|_{L^2(\Strip)}^2 + \gamma \int_0^1 |u_n(0,y)|^2 \, dy \right) \\
           & \to 2 E^{1D}_\gamma(\phi_{\tilde\omega,\gamma}) + \frac{2}{p+1} \| \tilde{\phi}_{\tilde\omega,\gamma} \|_{L^2(\Strip)}^2 - \frac{p-1}{p+1} \gamma  |\tilde{\phi}_{\tilde\omega,\gamma}(0,y)|^2 =   \frac{p-1}{p+1} \| \partial_x \tilde{\phi}_{\tilde\omega,\gamma} \|_{L^2(\Strip)}^2
       \end{aligned}
   \end{equation*}
    as $n \to \infty$. Thus we obtain \eqref{eqConverg1}. 
\end{proof}

We now show that, for $n$ large enough, $u_n$ does not depend on the transverse variable $y$.

\begin{lem}\label{lemIndip}
 Let $\gamma < 0$. There exist a subsequence of $(u_n)$, still denoted by $(u_n)$, and $N^*\in\mathbb N$ such that for any $n\geq N^*$ 
$u_n$ does not depend on the transverse variable $y$.  
\end{lem}
\begin{proof}
    For any $n>0$, $u_n$ is  $C^2$ on the strip, except when $x=0$, the traces $u_n(\cdot,0)$ and $u_n(\cdot,1)$ are $C^1$, and $\partial_yu_n(\cdot,0)=\partial_yu_n(\cdot,1)=0$ (see Appendix \ref{sec:properties} for a precise statement and proof).

    First, we notice the following 
    \begin{equation*}
    \begin{aligned}
        \dual{ \deltastripunit u_n}{ -\partial_{yy} u_n } &= -\int_0^1 u_n(0,y) \partial_{yy} u_n (0,y) dy = \int_0^1 |\partial_y u_n(0,y)|^2 dy \\ 
        & \ \ \ + u_n(0,0) \partial_y u_n(0,0) -  u_n(0,1) \partial_y u_n(0,1).
    \end{aligned}
    \end{equation*}
    By regularity  of $u_n$ we obtain that $u_n(0,0) \partial_y u_n(0,0) -  u_n(0,1) \partial_y u_n(0,1) = 0$. 
    Let $w_n = \partial_y u_n$. Taking the duality product of \eqref{eqUnEq1} with $-\partial_{yy} u_n$, we see that 
    \begin{equation}
        \label{eq3Pieces}
        \begin{aligned}
        0 &= \left(\frac{1}{L_n^2} - 1\right) \| \partial_y w_n \|_{L^2(\Strip)}^2  + \left(\omega_n \| w_n \|_{L^2(\Strip)}^2 + \| \nabla w_n \|_{L^2(\Strip)}^2 + \gamma \int_0^1 |w_n(0,y)|^2 dy\right) \\ 
        & \ \  - p \int_0^1 \int_\R |\tilde \phi_{\tilde\omega,\gamma}|^{p-1} | \partial_y u_n |^2 \, dx \, dy \\
        & \ \  - p \int_0^1 \int_\R (| u_n|^{p-1} - |\tilde \phi_{\tilde\omega,\gamma}|^{p-1}) | \partial_y u_n |^2 \, dx \, dy .
        \end{aligned}
    \end{equation}
    Notice that the terms in the above equation are well-defined because of the regularity of $u_n$. We now show that $u_n$ does not depend on $y$ for $n$ large enough.
    For the first line on the right-hand side of \eqref{eq3Pieces}, we notice that, due to \eqref{eqCoercL2} and $\omega_n >  \gamma^2/4$ for Proposition \ref{prop:non-existence}, we obtain 
     \begin{equation*}
        \omega_n \| w_n \|_{L^2(\Strip)}^2 + \| \nabla w_n \|_{L^2(\Strip)}^2 + \gamma \int_0^1 |w_n(0,y)|^2 dy \geq C \| w_n \|_{L^2(\Strip)}^2
    \end{equation*}
    for some $C > 0$. 
    For the second line, we get 
    \begin{equation*}
        - p \int_0^1 \int_\R |\tilde \phi_{\tilde\omega,\gamma}|^{p-1} | \partial_y u_n |^2 \, dx \, dy \geq - p \|\tilde \phi_{\tilde\omega,\gamma}\|^{p-1}_{L^\infty(\R)} \| w_n \|_{L^2(\Strip)}^2.
    \end{equation*}
    Since $\partial_y u_n = 0$ on the boundaries of the strip,  by Poincar\'e inequality we have $\| w_n \|_{L^2(\Strip)} \lesssim \| \partial_y w_n \|_{L^2(\Strip)}$ and the first two lines in \eqref{eq3Pieces} are eventually positive. 
    
    Finally, the third line in \eqref{eq3Pieces} converges to zero as we have 
    \begin{multline*}
         \int_0^1 \int_\R (| u_n|^{p-1} - |\tilde \phi_{\tilde\omega,\gamma}|^{p-1}) | \partial_y u_n |^2 \, dx \, dy \\
         \leq \| u_n - \phi_{\tilde{\omega},\gamma} \|_{L^{p+1}(\Strip)} (\| u_n \|_{L^{p+1}(\Strip)}^{p-2} + \| \phi_{\tilde{\omega},\gamma} \|_{L^{p+1}(\Strip)}^{p-2} )  \| \partial_y u_n \|_{L^{p+1}(\Strip)}^2  \\
          \lesssim \| u_n - \phi_{\tilde{\omega},\gamma} \|_{L^{p+1}(\Strip)}  \| \partial_y u_n \|_{H^1(\Strip)}^2
    \end{multline*}
    when $p \geq 2$ or 
    \begin{multline*}
         \int_0^1 \int_\R(| u_n|^{p-1} - |\tilde \phi_{\tilde\omega,\gamma}|^{p-1}) | \partial_y u_n |^2 \, dx \, dy  \leq \| u_n - \phi_{\tilde{\omega},\gamma} \|_{L^{p+1}(\Strip)}^{p-1}  \| \partial_y u_n \|_{L^{p+1}(\Strip)}^2 \\ 
         \lesssim \| u_n - \phi_{\tilde{\omega},\gamma} \|_{L^{p+1}(\Strip)}^{p-1}  \| \partial_y u_n \|_{H^1(\Strip)}^2
    \end{multline*}
    for $p <2$. In particular, the sum between the second and the third line is positive for $n$ large enough, due to the convergence $u_n \to \phi_{\tilde{\omega},\gamma}$ in $H^1(\Strip)$. Therefore, we have $w_n = \partial_y u_n = 0$ for $n$ large enough which concludes the proof. 
\end{proof}

We will now prove the main results of this section. 
\begin{proof}[Proof of Theorem \ref{thmShrink}] 
We first consider the normalized problem \eqref{eqEtilde}. 
We define 
\begin{equation*}
    L^* = \sup \left\{L >0 :  u_{L'} \mbox{ is independent of $y$ for any $L'\in[0,L]$}\right\}.
\end{equation*}
    Then $L^* >0$. Indeed, by contradiction, suppose that $L^* = 0$. Then there exists $(L_n) \subset \R^+$, $L_n \to 0^+$ such that $ u_{L_n}$ depends on $y$ in a nontrivial way for every $n\in\mathbb N$. This contradicts Lemma \ref{lemIndip}. 
    We then consider the problem with generic mass constraint, defined by
    \begin{equation}\label{eqEtildem}
    \tilde e_{m,\gamma}(L) = \inf \{ \tilde {E}_{L} (u): u \in H^1(\Strip), \, \| u \|_{L^2(\Strip)}^2 = m \}.
\end{equation}
Denote  $u_{L,m}$   non-negative minimizers to \eqref{eqEtildem}, $\tilde \phi_{m}(x,y) = \phi_{\omega_m, \gamma}(x)$  the extended one-dimensional soliton (where  $ \phi_{\omega_m, \gamma}$ is defined in \eqref{eqGS1DIntr} where $\omega_m$ is chosen so that 
    $
        \int_\R |\phi_{\omega_m, \gamma}(x) |^2 dx = m 
        $).
    By replicating the same arguments as when $m=1$, 
    we obtain that $ \| u_{L,m} - \tilde \phi_{m} \|_{H^1(\Strip)} \to 0$ (up to subsequences) and there exists $L^*(m) >0 $ such that for $0<L< L^*(m)$, $\| \partial_y u_{L,m} \|_{L^2} = 0$.
    
    Finally, we come back to the original problem. Let $\tilde m>0$ and $L>0$. Observe that 
    the minimization problem $e_{\tilde mL,\gamma}$ (defined in \eqref{eqMinEnergyMass}) verifies
    \[
    e_{\tilde mL,\gamma}=L \tilde e_{\tilde m,\gamma}(L).
    \]
    The minimizers $u_{L,\tilde m}$ of $\tilde e_{\tilde m,\gamma}$ can be rescaled to obtain  minimizers of $  e_{\tilde mL,\gamma}$ by defining $\psi_L(x,Ly) = u_{L,\tilde m}(x,y)$.
Therefore, if $L<L^*$, there exists $\psi_L:\R\to\R$ such that for any $(x,y)\in\Strip_L$ we have $\tilde \psi_L(x,y) = \psi_L(x)$ and
    \begin{equation*}
        \int_0^L \int_\R |\tilde \psi_L(x,y)|^2 dx dy = L  \int_\R |\psi_L(x)|^2 dx = \tilde m L. 
    \end{equation*}
    In particular, the one-dimensional function $\psi_L$ satisfies the stationary equation \eqref{eq1DEqIntr}, is positive and
    \begin{equation*}
         \int_\R |\psi_L(x)|^2 dx = \tilde m.
    \end{equation*}
    Thus, by uniqueness of positive solutions,  we have $\psi_L(x) = \phi_{\omega_{\tilde m}, \gamma}(x)$.
\end{proof}

We will now prove that for large $L$ the minimizers of the energy at fixed mass are necessarily $2$-$d$. 
Let $L^{**}$ be defined by
    \begin{equation*}
        L^{**} = \inf\{L>0 :\forall L'\in(0,L),\,\exists\,u_{L'}\text{ such that } \partial_y u_{L'} \neq 0\},
    \end{equation*}
    where $u_{L'}$ denote any minimizer of \eqref{eqEtilde}.
We show the following.

\begin{prop}
\label{prop:estimate-L**}
    For any $\gamma \in \R$, and any $m >0$, we have 
    \begin{equation*}
    \begin{aligned}
         (L^{**})^2 &\leq \frac{(p+1)m}{2} \left( \int_\R | \phi_{\omega_m, \gamma} |^{p+1}(x) dx \right)^{-1} \times \\ 
         & \times \inf \left\{ \frac{ \int_0^1 |\partial_y f(y)|^2 dy }{ \int_0^1 | f(y)|^{p+1} dy - 1} :  f \in H^1(0,1) \setminus \{1\}, \int_0^1 |f(y)|^2 dy = 1  \right\}
    \end{aligned}
    \end{equation*}
    where $ \phi_{\omega_m, \gamma}$ is the $1$-$d$ soliton defined in \eqref{eqGS1DIntr} such that $\|\phi_{\omega_m, \gamma}\|_{L^2(\R)}^2 = m$. 
\end{prop}

\begin{proof}
    Define $\psi(x,y) = \phi_{\omega_m, \gamma}(x) f(y) $ for some $f \in H^1(0,1)$, $f \not \equiv 1$ such that 
    \begin{equation}\label{eqFNorm1}
        \int_0^1 f^2(y) dy  = 1, 
    \end{equation}
    so that 
    \begin{equation*}
        \int_0^1 \int_\R |\psi(x,y)|^2 dx dy = \int_\R\phi_{\omega_m,\gamma}^2(x) dx  = m. 
    \end{equation*}
    Notice that 
    \begin{multline}
        \label{eqEnEstL1}
       2 E_{L,\gamma}(\psi) = \frac{1}{L^2} \int_0^1 |\partial_y f|^2 dy \int_\R |\phi_{\omega_m,\gamma}(x)|^2 dx  + 2 E^{1D}_{\gamma} (\phi_{\omega_m,\gamma}) \int_\R |f(y)|^2 dy \\
        - \frac{2}{p+1} \int_\R |\phi_{\omega_m,\gamma}|^{p+1} dx \left( \int_\R|f(y)|^{p+1} - |f(y)|^2 dy \right).
\end{multline}    
    Let $L_m(f) >0$ be such that $ E_{L_m(f)} (\psi) = E^{1D}_{\gamma} (\phi_{\omega_m,\gamma})$. From \eqref{eqFNorm1} and \eqref{eqEnEstL1}, we have the expression
    \begin{equation*}
      L^2_m(f) = \frac{m(p+1)}{2}   \left( \int_\R |\phi_{\omega_m,\gamma}|^{p+1} dx \right)^{-1} \left( \frac{\int_0^1 |\partial_y f|^2 dy}{\int_\R|f(y)|^{p+1} dy - 1} \right).
    \end{equation*}
    Moreover, if $L>L_m(f)$ (resp. $<$), then $ E_{L} (\psi) <E^{1D}_{\gamma} (\phi_{\omega_m,\gamma})$ (resp. $>$).
    The result follows by minimizing over every admissible $f$. Indeed $L^{**} \leq \inf_{f} L_m(f)$ as the Neumann boundary conditions and the positivity of $f$ follows from the minimization.    
\end{proof}

We may give a more precise estimate in a particular case. As the test function, we consider
\begin{equation*}
    f(y) = \sqrt{2} |\cos(2\pi y)|.
\end{equation*}
Then 
\begin{equation*}
    \frac{\int_0^1 |f'|^2dy}{\int_\R|f(y)|^{p+1} dy - 1} = \frac{4\pi^3}{2^\frac{p-1}{2} \int_0^{2\pi}|\cos(x)|^{p+1} dx - \pi }
\end{equation*}
For the cubic case $p = 3$, this leads to 
\begin{equation*}
    L^{**} \leq 16 \pi^2 m \left( \int_\R | \phi_{\omega_m, \gamma} |^{p+1}(x) dx \right)^{-1}.
\end{equation*}

\begin{proof}[Proof of Theorem \ref{thmShrink2}]
    The theorem is a direct corollary of Proposition \ref{prop:estimate-L**}. 
\end{proof}

\appendix
\section{Qualitatives properties of minimizers and solutions on the strip}
\label{sec:properties}

In this section, we establish some properties of the solutions to equation \eqref{eqDeltaStrip} and of the action and energy ground states.

We introduce the notation
\begin{equation}\label{eqTau0L}
    \tau_0 v(x,y) = v(x,0), \quad \tau_L v(x,y) = v(x,L),
\end{equation}
for $v \in C^\infty(\Strip_L)$. These are the projections on the boundaries of the strip $\Gamma_0 = \Strip \cap \{(x,y) \in \R^2 : y =0 \} $,  $\Gamma_L = \Strip \cap \{(x,y) \in \R^2 : y =L \} $.

\begin{proposition}
\label{prop:properties-sol}
Let $u\in H^1(\Strip_L)$ be a solution to \eqref{eqDeltaStrip}. Then $u$ satisfies the following properties. 
        \begin{enumerate}
        \item $u \in W^{3,q}(\dot{\Strip}_L)$ for any $q \in [2,\infty)$ where 
        \begin{equation*}
            \dot{\Strip}_L = \Strip_L \cap \{(x,y) \in \R^2 : x \neq 0, y \not\in \{0,L\} \}.
        \end{equation*}
        In particular, $u \in C^2(\dot{\Strip_L})$.
        \item We have $\tau_{0,L}u \in C^1(\Gamma_{0,L})$. 
        In particular $\partial_y u(0,0) =\partial_y u(0,L)  = 0$.
       \item  There exists $C>0$ such that
\[
|u(x,y)|\leq C \exp\left(-\sqrt{\omega}|x|\right)
\]
for any $(x,y)\in \Strip_L$.
    \end{enumerate}
\end{proposition}

In the case of ground states, we have the following (here, we simply refer to action ground states and energy ground states as ground states, as the arguments employed work indifferently in both cases).

\begin{prop}
\label{prop:properties-ground-state}
Let $\gamma< 0$.
Then for the minimization problems \eqref{eqActNehIntr} or \eqref{eqMinEnergyMass} there exists a ground state that is symmetric with respect to the lengthwise variable $x$, decreasing with respect to $x>0$, monotone in the transverse variable $y$, and strictly positive.
\end{prop}

\begin{proof}[Proof of Proposition \ref{prop:properties-sol}, parts (1) and (2)]   
For point (1), we use the standard bootstrap argument for the elliptic problems. Let us first separate the domain in the following way. We define $\Strip^{+} = \Strip \cap \{(x,y) \in \R^2 : x \geq 0\}$, and $\Strip^{-} = \Strip \cap \{(x,y) \in \R^2 : x \leq 0\}$. We will denote the corresponding interior parts as $\dot{\Strip}^{\pm}$. Let us focus on the $\Strip^{+}$ for the moment.
    In $\dot{\Strip}^+$, $u$ satisfies 
\begin{equation}
\label{eqEllEasy}
    \Delta u = |u|^{p-1} u - \omega u,
\end{equation}
    which means that the standard bootstrap argument (see for example in \cite{Ca03}) can be employed to obtain more regularity. We will just give the idea. Suppose that $u \in L^r(\dot{\Strip}^+)$ for some $r > p$. Then $u^{p}\in L^{q/p}(\dot{\Strip}^+)$ which implies that $u \in W^{2, q/p}(\dot{\Strip}^+)$ from \eqref{eqEllEasy}. By Sobolev's embedding theorem, and repeating this argument recursively, we obtain that $u \in W^{2,q}(\dot{\Strip}^+)$ for any $q \in [2,\infty)$. Then $u^{p}\in W^{1,q}(\dot{\Strip}^+)$ and we can bootstrap again to obtain that $u \in W^{3,q}(\dot{\Strip}^+)$ for any $q \in [2,\infty)$. In particular $u \in C^2(\dot{\Strip}^+)$. This proves (1).

    Now we extend the added regularity to the boundary. For this, we need the Trace theorem with a polygonal boundary, which can be found in \cite[Theorem 1.5.2.3]{Gr11}.  In particular, let us define $\Gamma_1^+ = \Strip^+ \cap \{(x,y) \in \R^2 : x = 0\}$, $\Gamma_2^+ = \Strip^+ \cap \{(x,y) \in \R^2 : y = 0\}$, $\Gamma_3^+ = \Strip^+ \cap \{(x,y) \in \R^2 : y = 1\}$ which form a polygonal boundary of $\dot{\Strip}^+$.  Then we can define the trace of $u_n$ by pieces $\tau: u_n \to (f_i)_{i=1,2,3}$ where $f_j = \tau_j u_n = u_{n|\Gamma_j}$. This map is such that $\tau: W^{1,p}(\dot{\Strip}^+) \to \prod W^{1-\frac{1}{p},p}(\Gamma_j^+)$ with the continuity property 
    \begin{equation*}
        f_1(0) = f_2(0), \quad f_1(L) = f_3(0)
    \end{equation*}
    as soon as $2 < p < \infty$. 
    Thus, in particular, we have that $\tau u_n \in \prod W^{3-\frac{1}{p},p}(\Gamma_j)$ for any $p$, and the point evaluation in the corners $(0,0)$ $(0,L)$ are well defined by $u_n(0,0) = f_2(0)$, $u_n(0,L) = f_3(0)$, and indeed the function 
    \begin{equation}
        v = f_1 \mathds{1}_{x \in \Gamma_1} + f_2 \mathds{1}_{x \in \Gamma_2} + f_3 \mathds{1}_{x \in \Gamma_3}
    \end{equation}
    is at least $C^1\left(\cup \Gamma_{j}^+ \right)$. 
    Finally, from the Neumann boundary conditions, we can conclude that $\partial_y u_n(0,0) =\partial_y u_n(0, L)  = 0$.
    Finally, the same argument works in $\Strip^{-}$. In this way, we obtain (2). This concludes the proof. 
\end{proof}

To prove the exponential decay in $x$ of solutions to \eqref{eqDeltaStrip}, we construct the Green function for the linear part of the equation, i.e. we find $g(x,y,\xi,\eta)$ such that  
\begin{equation}
    \label{eq:green_eq}
\begin{cases}
    -\partial_{xx}g-\partial_{yy}g+\gamma\deltastrip g+\omega g=\delta(x-\xi)\delta(y-\eta),\\
    \partial_yg(x,0,\xi,\eta)=
    \partial_yg(x,L,\xi,\eta)=0.
\end{cases}
\end{equation}
Here, by $\delta(x-\xi)$, we denote the distribution on $\Strip_L$ such that $\dual{\delta(x-\xi)}{\phi}=\int_0^L \phi(\xi,y)dy$, with a similar definition for $\delta(y-\eta)$. In particular, we have $\deltastrip =\delta(x-0)$. We assume that $\omega>0$ if $\gamma\geq 0$ and $\omega>\gamma^2/4$ if $\gamma<0$.

\begin{lem}
    The Green's function for \eqref{eq:green_eq} is given by 
    \begin{equation*}
g(x,y,\xi,\eta)=\sum_{k=0}^\infty g_k(x,\xi)\theta_k(y)\theta_k(\eta),
\end{equation*}
where the basis functions are defined by
\begin{equation}
    \label{eq:basis}
\begin{cases}
\lambda_k=\left(\frac{2k\pi}{L}\right)^2,\quad \theta_k(y)=\sqrt{\frac2L}\cos\left(\frac{2k\pi}{L}y\right),\quad k\geq 1,\\
\lambda_0=0,\quad \theta_0(y)=\sqrt{\frac1L},\quad k=0.
\end{cases}
\end{equation}
and the coefficients are given by
    \[
g_k(x,\xi)=\frac{1}{2\sqrt{\lambda_k+\omega}}\left(-\frac{\gamma}{\gamma+2\sqrt{\lambda_k+\omega}}e^{-\sqrt{\lambda_k+\omega}(|x|+|\xi|)}+e^{-\sqrt{\lambda_k+\omega}|x-\xi|}\right).
\]
\end{lem}

Numerous methods exist for the construction of Green's functions, see for example the recent book \cite{Du15}. Albeit the problem under consideration has a classical flavor, we have not found a previous construction of a Green's function in this particular case in the existing literature, and we give one in what follows. 

\begin{proof}
    We first decompose in the transverse direction along a suitable basis. Consider the operator $T:D(T)\subset L^2(0,T)\to L^2(0,T)$ defined by
\[
Tu=-\partial_{yy}u,\quad D(T)=\{ u\in H^2(0,L): \partial_yu(0)=\partial_yu(L)=0 \}.
\]
The operator $T$ is a self-adjoint operator with compact resolvent. Its eigenvalues and corresponding normalized eigenvectors are given for $k\in \mathbb N$ by \eqref{eq:basis}.
They form an orthonormal basis of $L^2(0,L)$.
We decompose $\delta(y-\eta)$ in this basis: 
\begin{equation*}
\delta(y-\eta)=\sum_{k=0}^\infty \theta_k(y)\theta_k(\eta).
\end{equation*}
This suggests searching for $g$ in the form
\begin{equation}
    \label{eq:def_g}
g(x,y,\xi,\eta)=\sum_{k=0}^\infty g_k(x,\xi)\theta_k(y)\theta_k(\eta),
\end{equation}
for $(g_k(x,\xi))_{k\in\mathbb N}\in \ell^2(\mathbb N)$. 
Inserting this decomposition in \eqref{eq:green_eq}, we are lead to search for $g_k$ such that
\begin{equation}
\label{eq:green_eq_k}
-\partial_{xx}g_k+\lambda_k g_k+\gamma\delta(x)g_k+\omega g_k=\delta(x-\xi).
\end{equation}
We construct solutions by treating each $\delta$ separately. Define $g_k^1$ and $g_k^2$ by
\[
g_k^1(x, \xi)=\frac{1}{2\sqrt{\lambda_k+\omega}}e^{-\sqrt{\lambda_k+\omega}|x|},\quad g_k^2(x, \xi)=\frac{1}{2\sqrt{\lambda_k+\omega}}e^{-\sqrt{\lambda_k+\omega}|x-\xi|}.
\]
They verify 
\[
-\partial_{xx}g^1_k+\lambda_k g^1_k+\omega g^1_k=\delta(x),\quad
-\partial_{xx}g^2_k+\lambda_k g^2_k+\omega g^2_k=\delta(x-\xi).
\]
Define $g_k$ by
\[
g_k=c^1_k g^1_k+g^2_k,\quad c^1_k=-\frac{\gamma g_k^2(0)}{1+\gamma g_k^1(0)},
\]
where $c^1_k$ has been constructed so as to have $c^1_k=-\gamma g_k(0)$. 
Note that the denominator in $c^1_k$ is non-zero since we have assumed that $\omega>0$ and $\omega>\gamma^2/4$ when $\gamma<0$.
The function $g_k$ verifies \eqref{eq:green_eq_k}, and the function $g$ defined in \eqref{eq:def_g} is the Green's function verifying \eqref{eq:green_eq}. 
\end{proof}

We now turn to exponential decay in $x$ of solutions to \eqref{eqDeltaStrip}.

\begin{proof}[Proof of Proposition \ref{prop:properties-sol}, part (3)]
Solutions to  \eqref{eqDeltaStrip} satisfy the integral identity 
\[
\int_{\Strip_L} g(x,y,\xi,\eta) |u(\xi,\eta)|^{p-1}u(\xi,\eta)d\xi d\eta=u(x,y).
\]
The spatial (in $x$) decay rate in such integral identity is given by the lowest decay rate of the two functions of the product. In the present case, as $p>1$, the decay rate of $u$ will be the one of $g$, which verifies 
\[
|g(x)|\leq C \exp\left(-\sqrt{\omega}|x|\right)
\]
as $|x|\to\infty$. Hence solutions of \eqref{eqDeltaStrip} are exponentially decaying at  $|x|\to\infty$, with decay rate $\sqrt{\omega}$.
\end{proof}

We now establish the additional properties for ground states. 

\begin{proof}[Proof  of Proposition \ref{prop:properties-ground-state}]
    If $u$ is a minimizer for $s_{\omega, \gamma}$ then $I_{\omega,\gamma}(u) = 0$ and $|u|$ is also a minimizer see Lemma \ref{lem:equiv-min-prob1}. In the same way, if $u$ is a minimizer for $e_{m,\gamma}$, then $|u|$ is also a minimizer. Thus let us suppose that $u \geq 0$. 

    We can perform the Steiner symmetrization in $x$ and monotone rearrangement in $y$. 
    For the first, we define 
    \begin{equation*}
        u^*(x,y) = \sup\{ t: \left| \{z: u(z,y) >t \} \right| > |x| \}
    \end{equation*}
where $\left| \{z: u(z,y) >t \} \right|$ denotes the Lebesgue measure of the set. Then $u^*$ is symmetric with respect to $x$, has the maximum in $x = 0$ and $\partial_x u^* \geq 0$ for $x >0$. It is well known, (see e.g. \cite{Li77}, \cite[Theorem 1]{BrZi87} or \cite[Theorem 2.1]{CiFu06}) that $\| u^* \|_{L^{p}(\Strip_{L})} = \| u \|_{L^{p}(\Strip_{L})}$ for any $p \in [2,\infty]$ while the P\'olya-Szeg\H{o} inequality is verified: $\| \nabla u^*\|_{L^2(\Strip_L)} \leq \| \nabla u\|_{L^2(\Strip_L)}$.
In particular, as $u^*$ is symmetric with respect to $x$ and $\| u^*(.,y) \|_{L^{\infty}(\R)} = u^{*}(0,y)$ and since $\gamma \leq 0$, we have $E(u^*) \leq E(u)$. 
For the monotone rearrangement in $y$, we may directly employ \cite[Theorem 2.8]{BeLR04}. Applying both the rearrangements, we obtain a symmetric and decreasing in $x$, monotone in $y$, and non-negative minimizer. 

The next step is to show the positivity of such a minimizer. Let us denote it again by $u$. Then by the Euler-Lagrange theory, $u$ satisfies \eqref{eqDeltaStrip}. 
In particular, for any point $(x,y)$ with $x \neq 0$, $y \not \in \{0,L\}$, $u$ satisfies the elliptic equation 
\begin{equation*}%
    \Delta u = |u|^{p-1} u - \omega u,
\end{equation*}
Thus, we directly obtain that $u > 0$ in this region by Harnack's inequality for harmonic functions.  Moreover, by symmetrization in $x$, we have that for any $y \in (0,L)$, $u(0,y) = \sup\{u(x,y), x \in \R \}$.  In particular, we get $u(0,y) >0$ which concludes the proof.  
\end{proof}

\bibliographystyle{abbrv} %
\bibliography{biblio}

\begin{thebibliography}{10}

\bibitem{AdCaFiNo14}
R.~Adami, C.~Cacciapuoti, D.~Finco, and D.~Noja.
\newblock Constrained energy minimization and orbital stability for the {NLS}
  equation on a star graph.
\newblock {\em Ann. Inst. Henri Poincar{\'e}, Anal. Non Lin{\'e}aire},
  31(6):1289--1310, 2014.

\bibitem{AdNoVi13}
R.~Adami, D.~Noja, and N.~Visciglia.
\newblock Constrained energy minimization and ground states for {NLS} with
  point defects.
\newblock {\em Discrete Contin. Dyn. Syst., Ser. B}, 18(5):1155--1188, 2013.

\bibitem{AkBaIbKi24}
T.~Akahori, Y.~Bahri, S.~Ibrahim, and H.~Kikuchi.
\newblock Pitchfork bifurcation at line solitons for nonlinear {S}chr\"odinger
  equations on the product space {$\mathbb R\times\mathbb T$}.
\newblock {\em Ann. Henri Poincar\'e}, 25(7):3467--3497, 2024.

\bibitem{BeCa81}
H.~Berestycki and T.~Cazenave.
\newblock {Instabilit\'e des \'etats stationnaires dans les \'equations de
  {S}chr\"odinger et de {K}lein-{G}ordon non lin\'eaires}.
\newblock {\em C. R. Acad. Sci. Paris}, 293(9):489--492, 1981.

\bibitem{BeLR04}
H.~Berestycki and T.~Lachand-Robert.
\newblock Some properties of monotone rearrangement with applications to
  elliptic equations in cylinders.
\newblock {\em Mathematische Nachrichten}, 266(1):3--19, 2004.

\bibitem{BoCa23}
F.~Boni and R.~Carlone.
\newblock N{LS} ground states on the half-line with point interactions.
\newblock {\em NoDEA Nonlinear Differential Equations Appl.}, 30(4):Paper No.
  51, 23, 2023.

\bibitem{Brezis2011}
H.~Brezis.
\newblock {\em {Functional analysis, Sobolev spaces and partial differential
  equations}}, volume 2,3.
\newblock Springer, 2011.

\bibitem{BrLi83}
H.~Brezis and E.~Lieb.
\newblock {A relation between pointwise convergence of functions and
  convergence of functionals}.
\newblock {\em Proceedings of the American Mathematical Society},
  88(3):486--490, 1983.

\bibitem{BrZi87}
J.~E. Brothers and W.~P. Ziemer.
\newblock Minimal rearrangements of {S}obolev functions.
\newblock {\em Acta Univ. Carolin. Math. Phys.}, 28(2):13--24, 1987.
\newblock 15th winter school in abstract analysis (Srn\'i, 1987).

\bibitem{Ca03}
T.~Cazenave.
\newblock {\em Semilinear {Schr\"odinger} Equations}, volume~10.
\newblock American Mathematical Soc., 2003.

\bibitem{CaLi82}
T.~Cazenave and P.-L. Lions.
\newblock {Orbital stability of standing waves for some nonlinear
  Schr{\"o}dinger equations}.
\newblock {\em Communications in Mathematical Physics}, 85:549--561, 1982.

\bibitem{CiFu06}
A.~Cianchi and N.~Fusco.
\newblock {Steiner symmetric extremals in P{\'o}lya--Szeg{\"o}-type
  inequalities}.
\newblock {\em Advances in Mathematics}, 203(2):673--728, 2006.

\bibitem{DeDoGaSe23}
C.~De~Coster, S.~Dovetta, D.~Galant, and E.~Serra.
\newblock On the notion of ground state for nonlinear {Schr{\"o}dinger}
  equations on metric graphs.
\newblock {\em Calc. Var. Partial Differ. Equ.}, 62(5):28, 2023.
\newblock Id/No 159.

\bibitem{deGrSm24}
A.~de~Laire, P.~Gravejat, and D.~Smets.
\newblock {Minimizing travelling waves for the Gross-Pitaevskii equation on
  $\mathbb R\times \mathbb T$}.
\newblock {\em Annales de la Facult{\'e} des Sciences de Toulouse.
  Math{\'e}matiques.}, to appear, 2024.

\bibitem{DoSeTi23}
S.~Dovetta, E.~Serra, and P.~Tilli.
\newblock Action versus energy ground states in nonlinear {Schr{\"o}dinger}
  equations.
\newblock {\em Math. Ann.}, 385(3-4):1545--1576, 2023.

\bibitem{Du15}
D.~G. Duffy.
\newblock {\em Green's functions with applications}.
\newblock Adv. Appl. Math. (Boca Raton). Boca Raton, FL: CRC Press, 2nd ed.
  edition, 2015.

\bibitem{EvGa15}
L.~C. Evans and R.~F. Gariepy.
\newblock {\em Measure theory and fine properties of functions}.
\newblock Textb. Math. Boca Raton, FL: CRC Press, 2nd revised ed. edition,
  2015.

\bibitem{Ex08}
P.~Exner and O.~Post.
\newblock Quantum networks modeled by graphs.
\newblock {\em AIP Conference Proceedings}, 998(1):1--17, 2008.

\bibitem{FuJe08}
R.~Fukuizumi and L.~Jeanjean.
\newblock {Stability of standing waves for a nonlinear Schrodinger equation
  with a repulsive Dirac delta potential}.
\newblock {\em Discrete and Continuous Dynamical Systems}, 21(1):121, 2008.

\bibitem{FuOhOz08}
R.~Fukuizumi, M.~Ohta, and T.~Ozawa.
\newblock Nonlinear schr{\"o}dinger equation with a point defect.
\newblock In {\em Annales de l'Institut Henri Poincar{\'e} C, Analyse non
  lin{\'e}aire}, volume 25,5, pages 837--845. Elsevier, 2008.

\bibitem{GeLeRo23}
F.~Genoud, S.~Le~Coz, and J.~Royer.
\newblock A minimal mass blow-up solution on a nonlinear quantum star graph.
\newblock {\em arXiv:2302.09678}, 2023.

\bibitem{GoHoWe04}
R.~H. Goodman, P.~J. Holmes, and M.~I. Weinstein.
\newblock Strong {NLS} soliton--defect interactions.
\newblock {\em Phys. D: Nonlinear Phenomena}, 192(3-4):215--248, 2004.

\bibitem{Gr11}
P.~Grisvard.
\newblock {\em Elliptic problems in nonsmooth domains}, volume~69 of {\em
  Class. Appl. Math.}
\newblock Philadelphia, PA: Society for Industrial {and} Applied Mathematics
  (SIAM), reprint of the 1985 hardback ed. edition, 2011.

\bibitem{GuIn24}
S.~Gustafson and T.~Inui.
\newblock Threshold even solutions to the nonlinear {Schr{\"o}dinger} equation
  with delta potential at high frequencies.
\newblock {\em Discrete Contin. Dyn. Syst.}, 44(10):3135--3176, 2024.

\bibitem{JeLu22}
L.~Jeanjean and S.-S. Lu.
\newblock On global minimizers for a mass-constrained problem.
\newblock {\em Calc. Var. Partial Differ. Equ.}, 61(6):18, 2022.

\bibitem{JeTa05}
L.~Jeanjean and K.~Tanaka.
\newblock A positive solution for a nonlinear {Schr{\"o}dinger} equation on
  {{\(\mathbb R^N\)}}.
\newblock {\em Indiana Univ. Math. J.}, 54(2):443--464, 2005.

\bibitem{Ko00}
S.~Kosugi.
\newblock A semilinear elliptic equation in a thin network-shaped domain.
\newblock {\em Journal of the Mathematical Society of Japan}, 52(3):673--697,
  2000.

\bibitem{Ko02}
S.~Kosugi.
\newblock Semilinear elliptic equations on thin network-shaped domains with
  variable thickness.
\newblock {\em J. Differ. Equations}, 183(1):165--188, 2002.

\bibitem{CoFuFi08}
S.~Le~Coz, R.~Fukuizumi, G.~Fibich, B.~Ksherim, and Y.~Sivan.
\newblock {Instability of bound states of a nonlinear Schr{\"o}dinger equation
  with a Dirac potential}.
\newblock {\em Phys. D}, 237(8):1103--1128, 2008.

\bibitem{Li77}
E.~H. Lieb.
\newblock Existence and uniqueness of the minimizing solution of {Choquard's}
  nonlinear equation.
\newblock {\em Studies in Applied Mathematics}, 57(2):93--105, 1977.

\bibitem{MaMu24}
M.~Mariş and A.~Mur.
\newblock Periodic traveling waves for nonlinear schr\"odinger equations with
  non-zero conditions at infinity in {$\mathbb R^2 $}.
\newblock {\em arXiv:2404.11772}, 2024.

\bibitem{Po12}
O.~Post.
\newblock {\em Spectral analysis on graph-like spaces}, volume 2039 of {\em
  Lect. Notes Math.}
\newblock Berlin: Springer, 2012.

\bibitem{TeTzVi14}
S.~Terracini, N.~Tzvetkov, and N.~Visciglia.
\newblock {The nonlinear Schr{\"o}dinger equation ground states on product
  spaces}.
\newblock {\em Analysis \& PDE}, 7(1):73--96, 2014.

\bibitem{Wi96}
M.~Willem.
\newblock {\em Minimax theorems}, volume~24 of {\em Prog. Nonlinear Differ.
  Equ. Appl.}
\newblock Boston: Birkh{\"a}user, 1996.

\bibitem{Ya15}
Y.~Yamazaki.
\newblock Stability of line standing waves near the bifurcation point for
  nonlinear {S}chr\"odinger equations.
\newblock {\em Kodai Math. J.}, 38(1):65--96, 2015.

\end{thebibliography}

\end{document}